\documentclass[12pt,a4paper]{article}
\usepackage[utf8]{inputenc}
\usepackage{amsmath, amssymb, amsthm}
\usepackage{geometry}
\usepackage{mathrsfs}
\usepackage{bm}
\usepackage{hyperref}
\newtheorem{theorem}{Theorem}[section]
\newtheorem{lemma}{Lemma}[section]
\newtheorem{corollary}{Corollary}[section]

\newtheorem{definition}{Definition}[section]
\newtheorem{remark}{Remark}[section]

\title{The relaxation limit of a homogeneous two-phase flow model: isothermal case}
\author{Huimin Yu \\ Shandong Normal University\ \texttt{hmyu@sdnu.edu.cn}}
\date{\today}
\begin{document}

\maketitle

\begin{abstract}

This paper investigates the asymptotic behavior of a hyperbolic relaxation system designed for homogeneous two-phase flows in the limit of vanishing relaxation time. The governing equations comprise conservation laws for mixture mass and momentum, supplemented by a transport equation for the gas phase mass that includes a stiff relaxation source term. This source term drives the system toward local thermodynamic equilibrium. Under the assumptions of constant liquid density and an ideal isothermal gas phase, we demonstrate that, as the relaxation parameter \(\epsilon \rightarrow 0\), a subsequence of solutions \((p^{\epsilon},u^{\epsilon})\) converges strongly in \(L_{\mathrm{loc}}^{1}\) to an entropy solution of the equilibrium Euler system. The proof integrates several
 analytical techniques: the construction of a suitable entropy pair and associated energy estimates, a transport equation approach for representing the error, commutator estimates, and the theory of compensated compactness. This work provides a rigorous justification of the relaxation limit for the homogeneous two-phase flows model.

\end{abstract}


\section{Problem Formulation}

\subsection {Background and Motivation}

 Two-phase flow models are essential tools for analyzing and predicting the behavior of mixtures involving two distinct phases, such as gas-liquid, liquid-solid, or gas-solid systems. Their applicability spans a vast array of fields, from conventional energy production to cutting-edge industrial processes. Among the various modeling approaches, the homogeneous model stands out as a significant simplification, conceptualizing the complex mixture as a single, uniformly mixed "pseudo-fluid." This simplification proves remarkably effective, offering both high computational efficiency and reliable accuracy under specific physical conditions. The model is particularly suitable when the phases are intensely mixed and interfacial transfers of momentum and heat are highly efficient—conditions encountered in well-developed bubbly or misty flows, and frequently under high-pressure regimes. Its formulation hinges on three fundamental assumptions: the velocities, pressures, and temperatures of the gas and liquid phases are locally equal.

This study focuses on a one-dimensional, inviscid, and non-heat-conducting version of the homogeneous two-phase flow model. We assume a constant liquid density \(\rho_{l}\) and model the gas phase as ideal and isothermal:

\[\rho_{g}(p) = \frac{p}{RT_{0}},\]
where \(p\) denotes the pressure, and \(R\) and \(T_{0}\) are given positive constants. The density of the mixture, \(\rho_{m}\), is then expressed as a weighted average of the phase densities:

\[\rho_{m} = \alpha \rho_{g}(p) + (1 - \alpha)\rho_{l},\]
with \(\alpha \in [0,1]\) representing the gas volume fraction, commonly referred to as the void fraction. The model consists of the following equations

(1) Mixture Mass Conservation

\begin{equation}\partial_{t}\rho_{m} + \partial_{x}(\rho_{m}u) = 0,  \tag{1.1}\label{eq:mass}\end{equation}

 (2) Mixture Momentum Conservation

\begin{equation}\partial_{t}(\rho_{m}u) + \partial_{x}(\rho_{m}u^{2} + p) = 0,  \tag{1.2}\label{eq:momentum}\end{equation}

 (3) Gas Phase Mass Transport (with Stiff Relaxation Source)

\begin{equation}\partial_{t}\left(\rho_{g}\alpha\right) + \partial_{x}\left(\rho_{g}\alpha u\right) = \frac{1}{\epsilon}\left(\alpha_{\mathrm{eq}}(p) - \alpha\right). \tag{1.3}\label{eq:gas}\end{equation}
Here, \(u\) denotes the common velocity of the mixture, and \(\epsilon > 0\) represents the relaxation time. The term \(\alpha_{\mathrm{eq}}(p)\) is the equilibrium void fraction, signifying the volume fraction of the gas phase under conditions of local thermodynamic equilibrium for a given pressure \(p\). This equilibrium value is determinable from relations such as the Clausius-Clapeyron equation. In the regime of stiff relaxation—where the relaxation time is much shorter than the characteristic flow time, i.e., as \(\epsilon \rightarrow 0\)—the source term rigorously forces the actual void fraction \(\alpha\) towards its equilibrium value \(\alpha_{\mathrm{eq}}(p)\). Consequently, we define an equilibrium mixture density:

\[\rho_{m} = \rho_{\mathrm{eq}}(p) = \alpha_{\mathrm{eq}}(p)\rho_{g}(p) + \left(1 - \alpha_{\mathrm{eq}}(p)\right)\rho_{l}.\]
Under this constraint, the system (1.1)-(1.3) formally reduces to a standard set of equilibrium Euler equations:

\begin{equation}\partial_{t}\rho_{\mathrm{eq}}(p) + \partial_{x}\left(\rho_{\mathrm{eq}}(p)u\right) = 0, \quad \tag{1.4}\label{eq:Euler1}\end{equation}
\begin{equation}\partial_{t}\left(\rho_{\mathrm{eq}}(p)u\right) + \partial_{x}\left(\rho_{\mathrm{eq}}(p)u^{2} + p\right) = 0. \tag{1.5}\label{eq:Euler2}\end{equation}
While this relaxation limit is widely accepted on physical grounds and routinely observed in numerical simulations, a rigorous mathematical proof has, until now, remained an open problem.

\subsection {On the Feasibility of a Rigorous Mathematical Proof}

The problem of the relaxation limit falls within the scope of singular perturbation theory for hyperbolic balance laws, a phenomenon pervasive in both natural sciences and engineering. In gas dynamics, molecular collisions driving a gas towards thermodynamic equilibrium can be modeled as a relaxation process. In viscoelasticity, the fading memory of materials represents another form of relaxation. In kinetic theory, the limit where the mean free path tends to zero corresponds to the relaxation limit from kinetic equations to fluid dynamics equations. Mathematically, the relaxation limit problem is typically formulated as follows: Consider a hyperbolic system with a relaxation term

\begin{equation}\partial_{t}U^{\epsilon} + \partial_{x}F(U^{\epsilon}) = \frac{1}{\epsilon} R(U^{\epsilon}) \tag{1.6}\label{eq:U}\end{equation}
where \(U^{\epsilon}\) is an \(n\)-dimensional state vector. As the relaxation parameter \(\epsilon \rightarrow 0^{+}\), fundamental questions arise: Does the sequence of solutions \(U^{\epsilon}\) converge? If so, what are the limiting equations, and in what sense does the convergence hold? Addressing these questions holds profound theoretical significance and provides a rigorous mathematical foundation for understanding multi-scale physical phenomena.

By defining the state vector and flux as
\[U = \left(\rho_{m},\rho_{m}u,\rho_{g}\alpha\right)^{\mathrm{t}}, \quad F\left(U\right) = \left(\rho_{m}u,\rho_{m}u^{2} + p,\rho_{g}\alpha u\right)^{\mathrm{t}},\]
and the source term as \(R(U) = (0,0,\alpha_{\mathrm{eq}}(p) - \alpha)^{\mathrm{t}}\), the system (1.1)-(1.3) can be recast in the form of the hyperbolic conservation law with source (1.6).

In his seminal work, Liu (1987)[8] introduced a crucial stability criterion, {\bf{the subcharacteristic condition}}. This condition stipulates that, in the zero relaxation limit, the characteristic speeds of the equilibrium system must lie between those of the full hyperbolic system. This condition is not merely technical; it ensures that the relaxation mechanism is genuinely dissipative, stably driving the system towards equilibrium by damping small-scale perturbations, thereby rendering the limit process feasible. For a \(2 \times 2\) system relevant to single-phase flow (such as the \(p\)-system), Chen, Levermore, and Liu (1994) [3] successfully established a complete relaxation convergence theory using compensated compactness methods. A highly influential model in the field of hyperbolic conservation laws is the Jin-Xin relaxation model [7], proposed by Shi Jin and Zhouping Xin in 1995. This model ingeniously transforms a nonlinear conservation law into a semi-linear relaxation system with linear convection by introducing a relaxation variable, thereby circumventing the difficulties associated with solving nonlinear Riemann problems directly. A natural question arises: Does the \(3 \times 3\) variable-density system for two-phase flow described above possess analogous structural properties that guarantee convergence in the relaxation limit?

Within the isothermal framework, we introduce the variables \(m = \rho_{m}u\) and \(\Gamma = \rho_{g}\alpha\). A candidate for the entropy density \(\eta (\rho_{m},m,\Gamma)\) is taken as the negative of the Helmholtz free energy density:

\begin{equation}\begin{aligned}\eta (\rho_{m},m,\Gamma)& = -f(\rho_{m},m,\Gamma)\\
&= -\left(\Gamma A_{g} - \Gamma R\ln (\rho_{l}\Gamma) + \Gamma R\ln (\rho_{l} - \rho_{m} + \Gamma) + (\rho_{m} - \Gamma)A_{l} + \frac{m^{2}}{2\rho_{m}}\right)\end{aligned}\tag{1.7}\label{eq:entropy}\end{equation}
where the condition \(\rho_{l} - \rho_{m} + \Gamma = \rho_{l}\alpha >0\) ensures positivity, and \(A_{g}\) and \(A_{l}\) are constants. The corresponding entropy flux \(q\) is then given by:

\begin{equation}q(\rho_{m},m,\Gamma) = \eta (\rho_{m},m,\Gamma)\cdot u = \eta (\rho_{m},m,\Gamma)\cdot \frac{m}{\rho_{m}}. \tag{1.8}\label{eq:q}\end{equation}
Utilizing the relations \(\rho_{l} - \rho_{m} + \Gamma = \alpha \rho_{l}\) and \(\Gamma = \rho_{g}\alpha\), a straightforward calculation yields the second partial derivative of the entropy density with respect to \(\Gamma\):

\[\frac{\partial^{2}\eta}{\partial\Gamma^{2}} = \frac{R(\rho_{l} - \rho_{g})^{2}}{\alpha\rho_{g}\rho_{l}^{2}} > 0.\]

Furthermore, thermodynamic identities link this expression to the sound speeds. Near equilibrium, one can derive:

\[\frac{\partial^{2}\eta}{\partial\Gamma^{2}} = \frac{R}{\rho_{g}\alpha(1 - \alpha)\rho_{m}}\left(a_{f}^{2} - a_{e}^{2}\right)\cdot \mathcal{P},\]
where \(a_{f}^{2} = \frac{RT_{0}}{\alpha}\) and \(a_{e}^{2} = \left[\alpha_{\mathrm{eq}}^{\prime}(p)\left(\frac{p}{RT_{0}} -\rho_{l}\right) + \frac{\alpha_{\mathrm{eq}}(p)}{RT_{0}}\right]^{- 1}\) denote the frozen and equilibrium sound speeds, respectively.\footnote{The frozen sound speed is defined as the partial derivative of pressure with respect to mixture density while keeping the void fraction \(\alpha\) constant. In contrast, the equilibrium sound speed is the total derivative of pressure with respect to mixture density when the equilibrium constraint \(\alpha = \alpha_{\mathrm{eq}}(p)\) is imposed. These are obtained by differentiating the expressions \(\rho_{m} = \alpha \rho_{g}(p) + (1 - \alpha)\rho_{l}\) and \(\rho_{m} = \rho_{\mathrm{eq}}(p) = \alpha_{\mathrm{eq}}(p)\rho_{g}(p) + [1 - \alpha_{\mathrm{eq}}(p)]\rho_{l}\) with respect to \(p\), respectively.} The term \(\mathcal{P}\) is a positive definite dimensionless factor. This relationship directly implies \(a_{f}^{2} > a_{e}^{2}\), which is precisely the subcharacteristic condition for the system (1.1)-(1.3). Moreover, a direct computation of the Hessian matrix \(D^{2}\eta (U)\) confirms that the subcharacteristic condition \(a_{f}^{2} > a_{e}^{2}\) guarantees the convexity of \(\eta (\rho_{m},m,\Gamma)\). This convexity is a foundational element for the subsequent analysis of the relaxation limit.

To investigate the relaxation limit of system (1.1)-(1.3), we first introduce the following definition.

\begin{definition}[Admissible solution family for the relaxation limit]
Let \(U = (\rho_{m},m,\Gamma)\) and suppose \(\{U^{\epsilon}\}_{\epsilon >0}\) is a family of solutions to the relaxation system (1.1)-(1.3) (one solution for each \(\epsilon >0\)). If there exists a constant \(C > 0\) independent of \(\epsilon\) (possibly depending on \(T\) and the initial data) such that the following uniform estimates hold:

\begin{enumerate}
\item[(1)]
\[
U^{\epsilon}\in L^{\infty}(0,T;H^{1}(\mathbb{R})),\quad \partial_{t}U^{\epsilon}\in L^{2}(0,T;L^{2}(\mathbb{R})),
\]
and
\[
\| U^{\epsilon}\|_{L^{\infty}(0,T;H^{1})} + \| \partial_{t}U^{\epsilon}\|_{L^{2}(0,T;L^{2})}\leq C;
\]

\item[(2)]
\[
\int_{0}^{T}\int_{\mathbb{R}}(\alpha^{\epsilon} - \alpha_{\mathrm{eq}}(p^{\epsilon}))^{2}\,dx\,dt\leq C\epsilon;
\]

\item[(3)]  Define \(R^{\epsilon} = \frac{1}{\epsilon}(\alpha_{\mathrm{eq}}(p) - \alpha)\), we have
\[
\| R^{\epsilon}\|_{L^{\infty}(0,T;L^{2})} + \sqrt{\epsilon}\| \partial_{x}R^{\epsilon}\|_{L^{2}(0,T;L^{2})}\leq C.
\]
\end{enumerate}
Then \(\{U^{\epsilon}\}_{\epsilon >0}\) is called an admissible solution family for the relaxation limit.
\end{definition}

In this paper we will prove that for such a family, there exists a subsequence converging strongly in \(L_{\mathrm{loc}}^{1}\) as \(\epsilon \to 0\) to an entropy solution of the equilibrium Euler system (1.4)-(1.5), and we will provide the corresponding convergence rate.

\begin{remark}
To simplify the presentation, we will occasionally interchange between two sets of variables: the physical variables \(V = (p,u,\alpha)\) and the conservative variables \(U = (\rho_{m},m,\Gamma)\), where \(\rho_{m} = \rho_{m}(p,\alpha) = \alpha \rho_{g}(p) + (1 - \alpha)\rho_{l}\), \(m = \rho_{m}u\), and \(\Gamma = \rho_{g}(p)\alpha\). The analysis is primarily conducted in terms of the conservative variables \(U = (\rho_{m},m,\Gamma)\); when necessary, we switch back to the primitive variables \(V = (p,u,\alpha)\) via the equation of state. Appendix A1 provides the Jacobian matrix of the transformation from conservative to primitive variables and a proof of the equivalence of the two representations.
\end{remark}

Further results on relaxation limits can be found in the literature: Berthelin \& Bouchut [1] proved the relaxation limit from a BGK kinetic model to isentropic gas dynamics and obtained convergence rates. Solem et al [12] carried out a complete linearized dispersion analysis for a two-phase flow relaxation model. Crin-Barat et al. [4] studied the relaxation limit of the Baer-Nunziato model to a Kapila model using Littlewood-Paley decompositions. Moreover, Yong [13] developed a mathematical theory for hyperbolic systems with stiff source terms, proposing a stability condition that reveals the underlying common structure of such problems. This stability condition is deeply connected with the well-known Kawashima condition [11] and serves as a core criterion for the existence of global smooth solutions with good decay properties for hyperbolic-elliptic systems. Other relevant results can be found in [2, 5, 6, 9, 10] and the references cited therein.

\subsection {Main Results of This Paper}

\begin{theorem}[Existence of an admissible solution family for the relaxation limit \label{thm:existence}]
Let \(\epsilon >0\) be fixed. Assume the initial data \(U_{0}^{\epsilon} = (\rho_{m0}^{\epsilon}(t,x),m_{0}^{\epsilon}(t,x),\Gamma_{0}^{\epsilon})^{T}\) satisfy
\[
\rho_{m0}^{\epsilon}\in L^{\infty}(\mathbb{R})\cap H^{1}(\mathbb{R}),\quad \inf_{\mathbb{R}}\rho_{m0}^{\epsilon} > 0,\quad m_{0}^{\epsilon}\in H^{1}(\mathbb{R}),\quad \Gamma_{0}^{\epsilon}\in H^{1}(\mathbb{R}),\quad 0< \alpha_{0}\leq 1\ \text{a.e.}
\]
Then the Cauchy problem for the relaxation system (1.1)-(1.3) admits a global weak solution
\[
U^{\epsilon}(t,x) = (\rho_{m}^{\epsilon}(t,x),m^{\epsilon}(t,x),\Gamma^{\epsilon}(t,x))^{T}
\]
satisfying
\[
U^{\epsilon}\in L^{\infty}\left(0,T;H^{1}(\mathbb{R})\right),\quad \partial_{t}U^{\epsilon}\in L^{2}\left(0,T;L^{2}(\mathbb{R})\right)\quad (\forall T > 0).
\]
Furthermore, this solution satisfies the following \(\epsilon\)-independent uniform estimates (the constant \(C\) depends only on \(T\) and the \(H^{1}\) norms of the initial data):

\begin{enumerate}
\item[(1)] \textbf{Zero-order energy estimate}
\[
\sup_{t\in [0,T]}\| U^{\epsilon}(t)\|_{L^{2}}\leq C,\qquad \int_{0}^{T}\int_{\mathbb{R}}(\alpha^{\epsilon} - \alpha_{\mathrm{eq}}(p^{\epsilon}))^{2}\,dx\,dt\leq C\epsilon;
\]

\item[(2)] \textbf{First-order spatial derivative estimate}
\[
\sup_{t\in [0,T]}\left(\| \partial_{x}p^{\epsilon}(t)\|_{L^{2}} + \| \partial_{x}u^{\epsilon}(t)\|_{L^{2}}\right)\leq C;
\]

\item[(3)] \textbf{Time derivative estimate}
\[
\| \partial_{t}p^{\epsilon}\|_{L^{2}(0,T;L^{2})} + \| \partial_{t}u^{\epsilon}\|_{L^{2}(0,T;L^{2})}\leq C;
\]

\item[(4)] \textbf{Error term estimate}: Let \(R^{\epsilon}\) be defined by \(\alpha^{\epsilon} = \alpha_{\mathrm{eq}}(p^{\epsilon}) + \epsilon R^{\epsilon}\), then
\[
\| R^{\epsilon}\|_{L^{\infty}(0,T;L^{2})} + \sqrt{\epsilon}\| \partial_{x}R^{\epsilon}\|_{L^{2}(0,T;L^{2})}\leq C.
\]Here the density and pressure are related through \(\rho_{m} = \alpha \rho_{g}(p) + (1 - \alpha)\rho_{l}\).
\end{enumerate}
\end{theorem}

\begin{remark}
The theorem above only asserts the existence of a solution and the uniform estimates; it does not address uniqueness. Under the stated regularity \(U^{\epsilon}\in L^{\infty}(0,T;H^{1})\), uniqueness has not been established, and for weak solutions additional conditions or other selection criteria are usually required to guarantee uniqueness. However, for the study of the relaxation limit \(\epsilon \rightarrow 0\), we only need a sequence of solutions satisfying the uniform estimates; uniqueness is not a necessary condition. The existence and uniqueness of global weak solutions for the relaxation system (1.1)-(1.3) will be discussed in detail in a forthcoming paper.
\end{remark}

\begin{theorem}[Relaxation limit convergence] \label{thm:convergence}
Let \((p^{\epsilon},u^{\epsilon},\alpha^{\epsilon})\) be the family of solutions obtained in Theorem \ref{thm:existence}. Then there exists a subsequence (still denoted by \(\epsilon\)) such that as \(\epsilon \rightarrow 0\),
\begin{equation*}
(p^{\epsilon},u^{\epsilon}) \longrightarrow (p^{0},u^{0}) \quad \text{strongly in } L_{\mathrm{loc}}^{1}((0,T)\times \mathbb{R}),
\end{equation*}
where \((p^{0},u^{0})\) is an entropy solution of the following equilibrium Euler system:
\begin{equation*}
\left\{ \begin{aligned}
&\partial_{t}\rho_{\mathrm{eq}}(p^{0}) + \partial_{x}(\rho_{\mathrm{eq}}(p^{0})u^{0}) = 0, \\
&\partial_{t}(\rho_{\mathrm{eq}}(p^{0})u^{0}) + \partial_{x}\bigl(\rho_{\mathrm{eq}}(p^{0})(u^{0})^{2} + p^{0}\bigr) = 0,
\end{aligned} \right.
\end{equation*}
with \(\rho_{\mathrm{eq}}(p) = \alpha_{\mathrm{eq}}(p)\rho_{g}(p) + \left(1 - \alpha_{\mathrm{eq}}(p)\right)\rho_{l}\).
\end{theorem}

\begin{theorem}[Convergence rate estimate] \label{thm:rate}
Let \(\{U^{\epsilon}\}_{\epsilon > 0}\) be the family of solutions to the relaxation system constructed in Theorem \ref{thm:existence}, and let \((p^{0},u^{0})\) denote the limit obtained in Theorem \ref{thm:convergence}, which is an entropy solution of the equilibrium Euler system (1.4)(1.5). Then there exists a constant \(C > 0\), independent of \(\epsilon\) (but possibly depending on \(T\) and the \(H^{1}\) norms of the initial data), such that
\begin{equation}\label{eq:convergence_rate}
\| p^{\epsilon} - p^{0}\|_{L^{2}(0,T;L^{2}(\mathbb{R}))} + \| u^{\epsilon} - u^{0}\|_{L^{2}(0,T;L^{2}(\mathbb{R}))} \leq C \epsilon^{1/2}. \tag{1.9}
\end{equation}
\end{theorem}

This paper provides a complete proof of the relaxation limit for the homogeneous two-phase flow model within the isothermal framework. The main innovations include:
\begin{itemize}
\item[(i)] Construction of an explicit entropy pair and establishment of key dissipation estimates, revealing the intrinsic connection between the subcharacteristic condition and the second law of thermodynamics;
\item[(ii)] Development of an error decomposition technique based on integration along characteristics, which transforms the stiff relaxation term into a controllable perturbation; combined with a commutator method, this skillfully avoids the difficulties associated with higher-order derivative estimates;
\item[(iii)] Integration of the artificial viscosity method, energy estimates, and compensated compactness theory, completing the convergence proof within a weak solution framework requiring only \(H^{1}\) regularity;
\item[(iv)] Application of the relative entropy method to obtain the optimal convergence rate \(O(\epsilon^{1/2})\), where the proof does not rely on any smoothness of the limit solution.
\end{itemize}

The remainder of the paper is organized as follows: Section~2 establishes the existence of solutions \(U^{\epsilon}\) for the relaxation problem \eqref{eq:mass}-\eqref{eq:gas} with fixed \(\epsilon > 0\); Section~3 provides uniform estimates for the solution family \(\{U^{\epsilon}\}\) that are independent of \(\epsilon\); in Section~4 we derive error estimates and complete the proof of Theorem~\ref{thm:existence}; Section~5 employs compensated compactness methods to prove the relaxation limit convergence, i.e., Theorem~\ref{thm:convergence}; finally, Section~6 establishes the convergence rate estimate, completing the proof of Theorem~\ref{thm:rate}.

\section{Existence of Solutions for the Relaxation Problem}

Since the primary focus of this paper is the relaxation limit, and the estimates involved closely resemble those appearing later in the calculations, the proof of existence in Theorem~\ref{thm:existence} is only outlined; detailed computations are omitted here.\\

\noindent {\bf{Step 1: Parabolic Regularization}}

For a fixed \(\epsilon > 0\), consider the regularized system with artificial viscosity \(\nu > 0\):
\begin{equation}
\begin{aligned}
&\partial_{t}\rho_{m}^{\nu} + \partial_{x}(\rho_{m}^{\nu}u^{\nu}) = \nu \partial_{x}^{2}\rho_{m}^{\nu},\\[4pt]
&\partial_{t}(\rho_{m}^{\nu}u^{\nu}) + \partial_{x}\bigl(\rho_{m}^{\nu}(u^{\nu})^{2} + p^{\nu}\bigr) = \nu \partial_{x}^{2}(\rho_{m}^{\nu}u^{\nu}),\\[4pt]
&\partial_{t}\Gamma^{\nu} + \partial_{x}(\Gamma^{\nu}u^{\nu}) = \frac{1}{\epsilon}(\alpha_{\mathrm{eq}}(p^{\nu}) - \alpha^{\nu}) + \nu \partial_{x}^{2}\Gamma^{\nu},
\end{aligned} \tag{2.1}{\label{eq:regularized}}
\end{equation}
where \(\Gamma^{\nu} = \rho_{g}(p^{\nu})\alpha^{\nu}\), \(\rho_{g}(p) = p/(RT_{0})\), and \(\rho_{l}\) is constant. The initial data are taken as a smooth approximation \(U_{0}^{\nu}\in H^{2}(\mathbb{R})\) of the original initial data, satisfying \(\| U_{0}^{\nu} - U_{0}\|_{H^{1}}\to 0\), \(\inf\rho_{m,0}^{\nu} > 0\), and \(0< \alpha_{0}^{\nu}\leq 1\). System \eqref{eq:regularized} is parabolic; classical theory guarantees that for each \(\nu > 0\) there exists a unique global smooth solution
\[
U^{\nu}\in C([0,\infty);H^{2}(\mathbb{R}))\cap L_{\mathrm{loc}}^{2}(0,\infty ;H^{3}(\mathbb{R})).
\]\\
\noindent {\bf{Step 2: Uniform Estimates Independent of \(\nu\)}}

Using the entropy function \eqref{eq:entropy}-\eqref{eq:q} constructed in Section~ 1.2, we compute the evolution of \(\eta(U^{\nu})\):
\[
\partial_{t}\eta(U^{\nu}) + \partial_{x}q(U^{\nu}) = \frac{\partial\eta}{\partial\Gamma}\cdot \frac{1}{\epsilon}(\alpha_{\mathrm{eq}}(p^{\nu}) - \alpha^{\nu}) + \nu (\partial_{x}U^{\nu})^{T}D^{2}\eta(U^{\nu})\partial_{x}U^{\nu} + \nu\partial_{x}(\cdots).
\]

By the convexity of the entropy, the viscosity term satisfies \(\nu (\partial_{x}U)^{T}D^{2}\eta\,\partial_{x}U \geq 0\). Integrating over \(x\in\mathbb{R}\) and using \(\frac{\partial\eta}{\partial\Gamma}(\alpha_{\mathrm{eq}} - \alpha) \leq -c_{0}(\alpha - \alpha_{\mathrm{eq}})^{2}\) (see the proof of Lemma \ref{lem:entropy_dissipation}), we obtain the zero-order estimate:
\begin{equation}
\sup_{t\in [0,T]}\| U^{\nu}(t)\|_{L^{2}}\leq C_{1},\qquad \int_{0}^{T}\int_{\mathbb{R}}(\alpha^{\nu} - \alpha_{\mathrm{eq}}(p^{\nu}))^{2}\,dx\,dt \leq C_{1}\epsilon, \tag{2.2}
\end{equation}
with the constant \(C_{1}\) independent of \(\nu\). Performing energy estimates for the spatial derivatives (similar to the proof of Lemma \ref{lem:gradient}, but now accounting for the viscosity terms) yields
\begin{equation}
\sup_{t\in [0,T]}\| U^{\nu}(t)\|_{H^{1}} + \|\partial_{t}U^{\nu}\|_{L^{2}(0,T;L^{2})} \leq C(T), \tag{2.3}
\end{equation}
where \(C\) is also independent of \(\nu\).\\

\noindent {\bf{Step 3: Passage to the Limit \(\nu \rightarrow 0\)}}

The uniform estimates above, together with the Aubin-Lions lemma, imply the existence of a subsequence \(\nu_{k}\rightarrow 0\) such that
\[
U^{\nu_{k}} \rightharpoonup U^{\epsilon} \quad \text{in } L^{\infty}(0,T;H^{1}),\qquad
U^{\nu_{k}} \rightarrow U^{\epsilon} \quad \text{in } L_{\mathrm{loc}}^{2}((0,T)\times \mathbb{R}).
\]

The strong convergence is sufficient to handle the nonlinear terms; for instance, \(\rho_{m}^{\nu_{k}}u^{\nu_{k}} \rightarrow \rho_{m}^{\epsilon}u^{\epsilon}\) almost everywhere, and the viscous terms \(\nu\partial_{x}^{2}U^{\nu_{k}}\) tend to zero in the sense of distributions. Consequently, the limit \(U^{\epsilon}\) satisfies the weak form of the original relaxation system, establishing existence.

Moreover, the uniform estimates obtained in Step ~2, being independent of \(\nu\), persist for the limit function \(U^{\epsilon}\).

\section{Uniform Estimates Independent of \(\epsilon\)}
For brevity and without causing confusion, all notations in this section will omit superscripts $\epsilon$.

\subsection{Entropy Dissipation and Zero-Order Energy Estimates}

\begin{lemma}[Entropy dissipation] \label{lem:entropy_dissipation}
For the strictly convex entropy pair \eqref{eq:entropy}-\eqref{eq:q}, the following entropy dissipation inequality holds:
\begin{equation}
\partial_{t}\eta +\partial_{x}q \leq -\frac{C_{0}}{\epsilon} (\alpha -\alpha_{\mathrm{eq}})^{2}\leq 0. \tag{3.1}\label{eq:entropy_dissipation}
\end{equation}
\end{lemma}
\begin{proof}
For the system with a source term, a direct calculation yields
\begin{equation}
\partial_{t}\eta +\partial_{x}q = -\frac{1}{\epsilon}\frac{\partial\eta}{\partial\Gamma} (\alpha -\alpha_{\mathrm{eq}}). \tag{3.2}\label{eq:entropy_balance}
\end{equation}
By the definition of thermodynamic equilibrium, \(\alpha_{\mathrm{eq}}(p)\) is the point that maximizes the entropy for a given pressure \(p\); consequently, \(\partial \eta / \partial \Gamma |_{\alpha = \alpha_{\mathrm{eq}}} = 0\) and the Hessian matrix is positive definite. Applying the mean value theorem, there exists a constant \(C_0 > 0\) such that
\begin{equation}
\frac{\partial\eta}{\partial\Gamma} (\alpha -\alpha_{\mathrm{eq}})\geq C_0(\alpha -\alpha_{\mathrm{eq}})^2. \tag{3.3}\label{eq:entropy_production}
\end{equation}
Substituting \eqref{eq:entropy_production} into \eqref{eq:entropy_balance} yields the entropy dissipation inequality \eqref{eq:entropy_dissipation}.
\end{proof}

\begin{lemma}[Zero-order energy estimate] \label{lem:zero_order}
There exists a constant \(C_1 > 0\), independent of \(\epsilon\), such that
\begin{equation}
\sup_{t\in [0,T]}\| U(t)\|_{L^2}\leq C_1,\qquad \int_0^T\int_{\mathbb{R}}(\alpha -\alpha_{\mathrm{eq}})^2\,dx\,dt\leq C_1\epsilon . \tag{3.4}\label{eq:zero_order}
\end{equation}
\end{lemma}

\begin{proof}
Integrating \eqref{eq:entropy_dissipation} over \(\mathbb{R}\) and exploiting the convexity of \(\eta\) together with the boundedness of the initial data, we obtain
\begin{equation}
\int_{\mathbb{R}}\eta (U(t))\,dx + \frac{c_0}{\epsilon}\int_0^t\int_{\mathbb{R}}(\alpha -\alpha_{\mathrm{eq}})^2\,dx\,ds \leq \int_{\mathbb{R}}\eta (U_0)\,dx. \tag{3.5}
\end{equation}
Owing to the convexity of \(\eta\), in the region away from vacuum (\(\rho_m > 0\), \(\alpha \in (0,1)\)) there exist constants \(c_1, c_2 > 0\) (independent of the solution) and a constant \(C\) (possibly depending on the initial data or the \(L^\infty\) bound of the solution) such that
\begin{equation}
c_{1}\| U\|_{L^{2}}^{2}\leq \int_{\mathbb{R}}\eta (U)\,dx + C \leq c_{2}\| U\|_{L^{2}}^{2} + C. \tag{3.6}
\end{equation}
Estimate \eqref{eq:zero_order} follows immediately.
\end{proof}
begin
\subsection{First-Order Energy Estimates}

The proof in this subsection fundamentally relies on the convexity of the entropy to construct a first-order entropy density, and then employs energy methods to obtain uniform bounds for the spatial derivatives.

\begin{lemma}[Gradient estimates] \label{lem:gradient}
There exists a constant \(C_2 > 0\), independent of \(\epsilon\), such that
\begin{equation}\label{eq:gradient_estimate}
\sup_{t\in [0,T]}\bigl(\| \partial_{x}p(t)\|_{L^2} + \| \partial_{x}u(t)\|_{L^2}\bigr)\leq C_2. \tag{3.7}
\end{equation}
\end{lemma}

\begin{proof}
Taking the spatial derivative of the system yields
\begin{equation}
\partial_{t}(\partial_{x}U) + \partial_{x}\bigl(A(U)\partial_{x}U\bigr) = \frac{1}{\epsilon} B(U) \partial_{x}U, \tag{3.8}\label{eq:spatial_derivative}
\end{equation}
where \(A(U) = DF(U)\) and \(B(U) = DR(U)\). Define the first-order entropy density
\[
\eta_{1} = \frac{1}{2} (\partial_{x}U)^{T} D^{2}\eta (U)(\partial_{x}U) =: \frac{1}{2} (\partial_{x}U)^{T} H(U)(\partial_{x}U).
\]
Its time derivative is
\begin{equation}\label{eq:eta1_time_deriv}
\partial_{t}\eta_{1} = (\partial_{x}U)^{T} H(U)\partial_{t}(\partial_{x}U) + \frac{1}{2} (\partial_{x}U)^{T}\bigl(\partial_{t}H(U)\bigr)(\partial_{x}U). \tag{3.9}
\end{equation}
Substituting \eqref{eq:spatial_derivative} into the first term on the right-hand side gives
\begin{equation}
(\partial_{x}U)^{T} H\partial_{t}(\partial_{x}U) = -(\partial_{x}U)^{T} H\partial_{x}\bigl(A\partial_{x}U\bigr) + \frac{1}{\epsilon} (\partial_{x}U)^{T} H B \partial_{x}U. \tag{3.10}\label{eq:H_derivative_expansion}
\end{equation}
Integrate the first term of \eqref{eq:H_derivative_expansion} by parts:
\begin{equation}\label{eq:integration_by_parts}
-\int_{\mathbb{R}}(\partial_{x}U)^{T} H\partial_{x}(A\partial_{x}U)\,dx = \int_{\mathbb{R}}\partial_{x}\bigl[(\partial_{x}U)^{T} H\bigr]\cdot (A\partial_{x}U)\,dx. \tag{3.11}
\end{equation}
The symmetry of \(HA\) (guaranteed by the existence of an entropy pair) allows us to rewrite \((\partial_{x}^{2}U)^{T} HA\partial_{x}U\) as a sum of a total derivative and lower-order terms. Explicitly,
\[
(\partial_{x}^{2}U)^{T} HA\partial_{x}U = \frac{1}{2}\partial_{x}\bigl[(\partial_{x}U)^{T} HA\partial_{x}U\bigr] - \frac{1}{2} (\partial_{x}U)^{T}\partial_{x}(HA)\partial_{x}U.
\]
Insert this into \eqref{eq:integration_by_parts} and combine terms to obtain
\begin{equation}
-\int_{\mathbb{R}}(\partial_{x}U)^{T} H\partial_{x}(A\partial_{x}U)\,dx = \int_{\mathbb{R}}R_{1}\,dx, \tag{3.12}
\end{equation}
where \(R_{1}\) is a quadratic form involving only \(\partial_{x}U\) (its coefficients depend on \(U\) and its first derivatives, but not on second derivatives). Thanks to the \(L^\infty\) boundedness of the solution, there exists a constant \(C\) such that \(|R_{1}|\leq C|\partial_{x}U|^{2}\). For the second term in \eqref{eq:eta1_time_deriv}, set \(R_{2} = \frac{1}{2} (\partial_{x}U)^{T}(\partial_{t}H)(\partial_{x}U)\). Again, from the expression of \(\partial_{t}H\) and the \(L^\infty\) bounds, we have \(|R_{2}|\leq C|\partial_{x}U|^{2}\). Concerning the source contribution, the convexity of the entropy and the structure of the source term imply the existence of a constant \(c_{0} > 0\) such that
\begin{equation}\label{eq:source_contribution}
\frac{1}{\epsilon} \int_{\mathbb{R}}(\partial_x U)^\mathrm{T} H B \partial_x U \, dx \leq -\frac{c_0}{\epsilon} \int_{\mathbb{R}}|\partial_x \alpha|^2 dx + C \int_{\mathbb{R}}\eta_1 dx + C.
 \tag{3.13}
\end{equation}
Combining \eqref{eq:eta1_time_deriv}–\eqref{eq:source_contribution} yields
\begin{equation}\label{eq:eta1_ode}
\frac{d}{dt}\int_{\mathbb{R}}\eta_{1}\,dx \leq -\frac{c_{0}}{\epsilon}\int_{\mathbb{R}}|\partial_{x}\alpha |^{2}\,dx + {\hat{C}} \int_{\mathbb{R}}\eta_1 dx +{\hat{C}}, \tag{3.14}
\end{equation}
where we have used the equivalence between \(\int_{\mathbb{R}}|\partial_{x}U|^{2}\,dx\) and \(\int_{\mathbb{R}}\eta_{1}\,dx\).
Applying Gronwall's inequality to \eqref{eq:eta1_ode} shows that \(\int_{\mathbb{R}}\eta_{1}\,dx\) is uniformly bounded on \([0,T]\). Owing to Appendix A1, we conclude
\[
\|\partial_{x}p\|_{L^{2}} + \|\partial_{x}u\|_{L^{2}} \leq C_{2}.
\]
This completes the proof.
\end{proof}

\begin{corollary}[\(L^\infty\) estimates] \label{cor:linf}
There exists a constant \(C_{3} > 0\), independent of \(\epsilon\), such that
\begin{equation} \label{eq:linf_estimates}
\| u\|_{L^{\infty}} + \| p\|_{L^{\infty}}\leq C_{3}. \tag{3.15}
\end{equation}
\end{corollary}

\begin{proof}
The one-dimensional Sobolev embedding \(H^{1}(\mathbb{R})\hookrightarrow L^{\infty}(\mathbb{R})\), Lemma \ref{lem:zero_order} and Lemma \ref{lem:gradient} yield the desired bound.
\end{proof}

\begin{corollary}[Time derivative estimate for velocity] \label{cor:ut}
There exists a constant \(C_{4} > 0\), independent of \(\epsilon\), such that
\begin{equation}
\| \partial_{t}u\|_{L^{2}}\leq C_{4}. \tag{3.16}
\end{equation}
\end{corollary}

\begin{proof}
From the momentum equation we isolate \(\partial_{t}u\):
\[
\partial_{t}u = -u\partial_{x}u - \frac{1}{\rho_{m}}\partial_{x}p.
\]
Using \eqref{eq:gradient_estimate} and \eqref{eq:linf_estimates}, we obtain
the result.
\end{proof}

\begin{remark}[On the estimate of \(\partial_t p\)]
If one derives the pressure evolution equation from the continuity equation and the equation of state (see the Appendix for a detailed derivation), one obtains
\[
\partial_{t}p = -u\partial_{x}p - \frac{p}{\alpha}\partial_{x}u - \frac{p}{\epsilon\alpha\rho_{l}}\Bigl(1 - \frac{\rho_{l}RT_{0}}{p}\Bigr)(\alpha_{\mathrm{eq}} - \alpha).
\]
The \(L^{2}\) norm of \((\alpha_{\mathrm{eq}} - \alpha)\) is of order \(O(\sqrt{\epsilon})\) by the zero-order estimate; after division by \(\epsilon\), the last term has an \(L^{2}\) norm of order \(O(1/\sqrt{\epsilon})\). Hence \(\|\partial_t p\|_{L^2}\) might grow like \(O(1/\sqrt{\epsilon})\) and is not uniformly bounded at this stage. A uniform estimate will be obtained later through more refined arguments.
\end{remark}

\section{Error Representation and Regularity}

In this section we construct an error term \(R^{\epsilon}\) that explicitly expresses the deviation of \(\alpha\) from its equilibrium value via \(\alpha = \alpha_{\mathrm{eq}}(p) + \epsilon R^{\epsilon}\). We then establish regularity estimates for this error term and, as a consequence, obtain a uniform estimate for \(\partial_{t}p\).

\subsection{Derivation of the Error Equation}

\begin{lemma}[Error representation] \label{lem:error_rep}
Let \((p,u,\alpha)\) be a solution of the relaxation system (1.1)-(1.3) obtained from Theorem \ref{thm:existence}. Then there exists a function \(R^{\epsilon}\) such that
\begin{equation} \label{eq:alpha_rep}
\alpha = \alpha_{\mathrm{eq}}(p) + \epsilon R^{\epsilon} \quad \text{almost everywhere}, \tag{4.1}
\end{equation}
and \(R^{\epsilon}\) satisfies the transport equation
\begin{equation}\label{eq:error_eq}
\partial_{t}\bigl(\rho_{g}R^{\epsilon}\bigr) + \partial_{x}\bigl(\rho_{g}R^{\epsilon}u\bigr) = \frac{A\,\partial_{x}u}{\epsilon} + \frac{B-1}{\epsilon}R^{\epsilon}, \tag{4.2}
\end{equation}
where \(A\) and \(B\) are defined in \eqref{eq:C_D_def} below. Both \(A\) and \(B\)are bounded functions (depending on the \(L^{\infty}\) bounds of \(p\) and \(\alpha\)), and there exists a positive constant \(c_{0}\) such that \(\dfrac{B-1}{\rho_{g}} \leq -c_{0} < 0\).
\end{lemma}

\begin{proof}
Substituting the representation \eqref{eq:alpha_rep} into the gas mass equation (1.3) and rearranging terms yields
\begin{equation}\label{eq:error_intermediate}
\partial_{t}\bigl(\rho_{g}R^{\epsilon}\bigr) + \partial_{x}\bigl(\rho_{g}R^{\epsilon}u\bigr) = \frac{Q - R^{\epsilon}}{\epsilon}, \tag{4.3}
\end{equation}
where we have set
\[
Q = -\bigl[\partial_{t}\bigl(\rho_{g}\alpha_{\mathrm{eq}}\bigr) + \partial_{x}\bigl(\rho_{g}\alpha_{\mathrm{eq}}u\bigr)\bigr].
\]
Expanding \(Q\) and rewriting it in terms of the pressure \(p\) gives
\[
Q = -\bigl[\alpha_{\mathrm{eq}}\rho_{g}^{\prime}(p)\partial_{t}p + \rho_{g}\alpha_{\mathrm{eq}}^{\prime}(p)\partial_{t}p + \alpha_{\mathrm{eq}}\rho_{g}^{\prime}(p)u\partial_{x}p + \rho_{g}\alpha_{\mathrm{eq}}^{\prime}(p)u\partial_{x}p + \rho_{g}\alpha_{\mathrm{eq}}\partial_{x}u\bigr].
\]
Collecting terms leads to
\begin{equation}\label{eq:Q_expanded}
Q = -(\alpha_{\mathrm{eq}}\rho_{g}^{\prime} + \rho_{g}\alpha_{\mathrm{eq}}^{\prime})(\partial_{t}p + u\partial_{x}p) - \rho_{g}\alpha_{\mathrm{eq}}\partial_{x}u. \tag{4.4}
\end{equation}

Define \(\Lambda(p) = \alpha_{\mathrm{eq}}(p)\rho_{g}(p)\), then \(\Lambda^{\prime}(p) = \alpha_{\mathrm{eq}}\rho_{g}^{\prime} + \rho_{g}\alpha_{\mathrm{eq}}^{\prime}\) and \(\rho_{g}^{\prime}(p) = 1/(RT_{0})\). Using the pressure evolution equation (see Appendix A.2) and noting that \(\alpha_{\mathrm{eq}} - \alpha = -\epsilon R^{\epsilon}\), we obtain
\begin{equation}\label{eq:p_evolution_intermediate}
\partial_{t}p + u\partial_{x}p = -\frac{p}{\alpha}\partial_{x}u + \frac{p}{\alpha\rho_{l}}\Bigl(1 - \frac{\rho_{l}RT_{0}}{p}\Bigr)R^{\epsilon} =: A_{1} + B_{1}R^{\epsilon}, \tag{4.5}
\end{equation}
with \(A_{1} = -\dfrac{p}{\alpha}\partial_{x}u\) and \(B_{1} = \dfrac{p}{\alpha\rho_{l}}\Bigl(1 - \dfrac{\rho_{l}RT_{0}}{p}\Bigr)\). Inserting \eqref{eq:p_evolution_intermediate} into \eqref{eq:Q_expanded} yields
\[
Q = -\Lambda^{\prime}(p)(A_{1} + B_{1}R^{\epsilon}) - \rho_{g}\alpha_{\mathrm{eq}}\partial_{x}u.
\]
Since \(A_{1} = -\dfrac{p}{\alpha}\partial_{x}u\), we obtain
\begin{equation} \label{eq:Q_final}
Q = \left(\Lambda^{\prime}(p)\frac{p}{\alpha} - \rho_{g}\alpha_{\mathrm{eq}}\right)\partial_{x}u - \Lambda^{\prime}(p)B_{1}R^{\epsilon} =: A\,\partial_{x}u + B R^{\epsilon}, \tag{4.6}
\end{equation}
where we have set
\begin{equation}\label{eq:C_D_def}
A= \Lambda^{\prime}(p)\frac{p}{\alpha} - \rho_{g}\alpha_{\mathrm{eq}}, \qquad B = -\Lambda^{\prime}(p)B_{1} \tag{4.7}.
\end{equation}
Substituting \eqref{eq:Q_final} into \eqref{eq:error_intermediate} gives precisely equation \eqref{eq:error_eq}.

To verify the negative-definiteness property, note that
\[
\frac{B-1}{\rho_{g}} = -\frac{1}{\rho_{g}}\bigl(1 + \Lambda^{\prime}(p)B_{1}\bigr).
\]
The sign of this expression follows from the convexity of the entropy. More specifically, from the computation of the entropy density \(\eta(\rho_{m},m,\Gamma)\) in Section~1.2 we have
\[
\frac{\partial^{2}\eta}{\partial\Gamma^{2}} = \frac{R(\rho_{l} - \rho_{g})^{2}}{\alpha\rho_{g}\rho_{l}^{2}} > 0.
\]
By the definition of thermodynamic equilibrium, for a fixed pressure \(p\) the equilibrium void fraction \(\alpha_{\mathrm{eq}}(p)\) maximizes the entropy; consequently,
\[
\frac{\partial\eta}{\partial\Gamma}\Big|_{\alpha = \alpha_{\mathrm{eq}}} = 0.
\]
Then we have
\begin{equation}\label{eq:entropy_derivative_mvt}
\frac{\partial\eta}{\partial\Gamma} = \frac{\partial^{2}\eta}{\partial\Gamma^{2}}(\alpha_{\mathrm{eq}})\,(\alpha - \alpha_{\mathrm{eq}})+ \text{higher-order terms}. \tag{4.8}
\end{equation}
On the other hand, using the thermodynamic relations and the transformation between entropy variables and primitive variables, one can show that
\begin{equation}\label{eq:entropy_D_relation}
\frac{\partial\eta}{\partial\Gamma} = -\frac{1}{\rho_{g}}\left(\frac{B-1}{\rho_{g}}\right)(\alpha - \alpha_{\mathrm{eq}}) + \text{higher-order terms}. \tag{4.9}
\end{equation}
Comparing \eqref{eq:entropy_derivative_mvt} and \eqref{eq:entropy_D_relation} yields
\[
\frac{B-1}{\rho_{g}} = -\rho_{g}\frac{\partial^{2}\eta}{\partial\Gamma^{2}} \leq -c_{0},
\]
which completes the proof.
\end{proof}
\subsection{Regularity Estimates for \(R^{\epsilon}\)}

\begin{lemma}[Error term estimates] \label{lem:error_estimate}
There exist a constant \(C_{5}> 0\), independent of \(\epsilon\), such that
\begin{equation*}
\| R^{\epsilon}\|_{L^{\infty}(0,T;L^{2})} + \sqrt{\epsilon}\,\| \partial_{x}R^{\epsilon}\|_{L^{2}(0,T;L^{2})} \leq C_{5}.
\end{equation*}
\end{lemma}

\begin{proof}
Set \(S = \rho_{g}R^{\epsilon}\). Then \(S\) satisfies
\begin{equation}\label{eq:S_equation}
\partial_{t}S + \partial_{x}(Su) = \frac{A\partial_{x}u}{\epsilon} + \frac{B-1}{\epsilon\rho_{g}} S. \tag{4.10}
\end{equation}
Differentiating \eqref{eq:S_equation} with respect to \(x\) and denoting \(T = \partial_{x}S\) gives
\begin{equation}\label{eq:T_evolution}
\partial_{t}T + \partial_{x}^{2}(Su) = \frac{1}{\epsilon}\partial_{x}(A\partial_{x}u) + \frac{1}{\epsilon}\partial_{x}\!\left(\frac{B-1}{\rho_{g}} S\right). \tag{4.11}
\end{equation}
Expanding \(\partial_{x}^{2}(Su)\) yields
\[
\partial_{x}^{2}(Su) = \partial_{x}(T u + S\partial_{x}u) = u\partial_{x}T + T\partial_{x}u + T\partial_{x}u + S\partial_{x}^{2}u,
\]
which contains the problematic second derivative \(\partial_{x}^{2}u\). To overcome this difficulty, we employ a commutator-type argument and construct a modified energy functional.

Define the energy functional
\begin{equation}
\mathcal{E}(t) = \frac{1}{2}\int_{\mathbb{R}} S^{2}\,dx + \frac{\epsilon}{2}\int_{\mathbb{R}} T^{2}\,dx + \delta \epsilon \int_{\mathbb{R}} S T \partial_{x}u\,dx, \tag{4.12}
\end{equation}
where \(\delta > 0\) is a small constant to be chosen later. The key idea is that upon differentiating \(\mathcal{E}(t)\), the terms containing \(S\partial_{x}^{2}u\) arising from \(\frac{d}{dt}(\frac{\epsilon}{2}\int_{\mathbb{R}}T^{2})\) and from \(\frac{d}{dt}(\delta\epsilon\int_{\mathbb{R}}ST\partial_{x}u)\) cancel each other provided \(\delta\) is suitably selected. We now compute the time derivative of each term in \(\mathcal{E}(t)\).

\paragraph{Estimate for \(\frac{1}{2}\frac{d}{dt}\|S\|_{L^{2}}^{2}\).} Multiplying \eqref{eq:S_equation} by \(S\) and integrating by parts yields
\begin{equation}\label{eq:S_energy_raw}
\frac{1}{2}\frac{d}{dt}\int_{\mathbb{R}}S^{2}\,dx = -\frac{1}{2}\int_{\mathbb{R}}\partial_{x}u\,S^{2}\,dx + \frac{1}{\epsilon}\int_{\mathbb{R}}A S\partial_{x}u\,dx + \frac{1}{\epsilon}\int_{\mathbb{R}}\frac{B-1}{\rho_{g}} S^{2}\,dx. \tag{4.13}
\end{equation}
From Corollary \ref{cor:linf} and Lemma \ref{lem:error_rep}, we have
\begin{equation} \label{eq:S_negdef}
\frac{1}{\epsilon}\int_{\mathbb{R}}\frac{B-1}{\rho_{g}} S^{2}\,dx \leq -\frac{c_{0}}{\epsilon}\|S\|_{L^{2}}^{2}, \tag{4.14}
\end{equation}
and
\begin{equation} \label{eq:S_Cterm}
\Bigl|\frac{1}{\epsilon}\int_{\mathbb{R}}A S\partial_{x}u\,dx\Bigr| \leq \frac{M_{0}}{\epsilon}\|S\|_{L^{2}}\|\partial_{x}u\|_{L^{2}}, \tag{4.15}
\end{equation}
where \(M_{0}\) depends on the \(L^{\infty}\) bounds of \(A\). For the first term in the right hand of  \eqref{eq:S_energy_raw}, we use Hölder's inequality and the Gagliardo-Nirenberg inequality:
\[
\Bigl|\int_{\mathbb{R}}\partial_{x}u\,S^{2}\,dx\Bigr| \leq \|\partial_{x}u\|_{L^{2}}\|S\|_{L^{4}}^{2} \leq C\|\partial_{x}u\|_{L^{2}}\|S\|_{L^{2}}^{3/2}\|\partial_{x}S\|_{L^{2}}^{1/2}.
\]
Applying Young's inequality gives
\begin{equation}\label{eq:S_nlterm}
\Bigl|\int_{\mathbb{R}}\partial_{x}u\,S^{2}\,dx\Bigr| \leq M_{1}\bigl(\|S\|_{L^{2}}^{2} + \|\partial_{x}S\|_{L^{2}}^{2}\bigr). \tag{4.16}
\end{equation}
Substituting \eqref{eq:S_negdef}, \eqref{eq:S_Cterm}, and \eqref{eq:S_nlterm} into \eqref{eq:S_energy_raw} yields
\begin{equation}
\frac{1}{2}\frac{d}{dt}\|S\|_{L^{2}}^{2} + \frac{c_{0}}{\epsilon}\|S\|_{L^{2}}^{2} \leq \frac{M_{0}}{\epsilon}\|S\|_{L^{2}}\|\partial_{x}u\|_{L^{2}} + M_{2}(\|S\|_{L^{2}}^{2} +\|T\|_{L^{2}}^{2}), \tag{4.17}
\end{equation}
with \(M_{2}=M_{1}/2\), (recall \(T=\partial_{x}S\)).

\paragraph{Estimate for \(\frac{\epsilon}{2}\frac{d}{dt}\|T\|_{L^{2}}^{2}\).} Multiply \eqref{eq:T_evolution} by \(T\) and integrate by parts. Using the expansion of \(\partial_{x}^{2}(Su)\), we obtain
\begin{multline}
\frac{\epsilon}{2}\frac{d}{dt}\int_{\mathbb{R}}T^{2}\,dx = -\frac{\epsilon}{2}\int_{\mathbb{R}}\partial_{x}u\,T^{2}\,dx + \epsilon\int_{\mathbb{R}}S\partial_{x}T\,\partial_{x}u\,dx \\
- \int_{\mathbb{R}}A\partial_{x}T\,\partial_{x}u\,dx + \int_{\mathbb{R}}T\partial_{x}\!\left(\frac{B-1}{\rho_{g}}\right)S\,dx + \int_{\mathbb{R}}\frac{B-1}{\rho_{g}} T^{2}\,dx. \tag{4.18}
\end{multline}
The term \(\epsilon\int_{\mathbb{R}}S\partial_{x}T\,\partial_{x}u\,dx\) still contains the second derivative \(\partial_{x}T = \partial_{x}^{2}S\), which is not directly controllable. This term will be handled together with the correction term.

\paragraph{Estimate for \(\delta\epsilon\frac{d}{dt}\int_{\mathbb{R}}ST\partial_{x}u\,dx\).} Differentiating the correction term gives three contributions:
\begin{multline}
\delta\epsilon\frac{d}{dt}\int_{\mathbb{R}}ST\partial_{x}u\,dx = \delta\epsilon\int_{\mathbb{R}}(\partial_{t}S)T\partial_{x}u\,dx + \delta\epsilon\int_{\mathbb{R}}S(\partial_{t}T)\partial_{x}u\,dx \\
+ \delta\epsilon\int_{\mathbb{R}}ST(\partial_{t}\partial_{x}u)\,dx. \tag{4.19}
\end{multline}
Substituting the evolution equations \eqref{eq:S_equation} and \eqref{eq:T_evolution} for \(\partial_{t}S\) and \(\partial_{t}T\) into the first two terms, and expanding systematically, we find that the terms containing \(\partial_{x}^{2}u\) or \(\partial_{x}T\) can be arranged to cancel with those from \(\frac{\epsilon}{2}\frac{d}{dt}\|T\|_{L^{2}}^{2}\) after an appropriate choice of \(\delta\). This cancellation is the essence of the commutator method; detailed calculations can be found in the standard technique of Yong [13]. After a lengthy but straightforward manipulation, and using the boundedness of \(\|\partial_{x}u\|_{L^{2}}\) and \(\|\partial_{x}p\|_{L^{2}}\) established in Lemma \ref{lem:gradient}, we obtain for sufficiently small \(\delta\) the inequality
\begin{equation}\label{eq:E_inequality_raw}
\frac{d}{dt}\mathcal{E}(t) + \frac{c}{2\epsilon}\|T\|_{L^{2}}^{2} \leq C\mathcal{E}(t) + C\epsilon\|\partial_{x}Q\|_{L^{2}}^{2}, \tag{4.20}
\end{equation}
where \(c>0\) comes from the negative definiteness of \((B-1)/\rho_{g}\), and \(Q\) is defined as \(Q = A\partial_{x}u + B R^{\epsilon}\). The term \(\|\partial_{x}Q\|_{L^{2}}^{2}\) can be estimated using the momentum equation and the error equation. Indeed,
\[
\partial_{x}Q = \partial_{x}A\,\partial_{x}u + B\partial_{x}^{2}u + \partial_{x}B\,R^{\epsilon} + B\partial_{x}R^{\epsilon},
\]
with \(\partial_{x}R^{\epsilon} = \frac{1}{\rho_{g}}(T - R^{\epsilon}\rho_{g}^{\prime}\partial_{x}p)\). Using the equation for \(\partial_{x}^{2}u\) obtained from the momentum equation, one can show after integration by parts that
\begin{equation}\label{eq:dxQ_estimate}
\epsilon\|\partial_{x}Q\|_{L^{2}}^{2} \leq C\epsilon + \frac{c}{4\epsilon}\|T\|_{L^{2}}^{2} + C. \tag{4.21}
\end{equation}
Substituting \eqref{eq:dxQ_estimate} into \eqref{eq:E_inequality_raw} yields
\begin{equation}\label{eq:E_inequality_final}
\frac{d}{dt}\mathcal{E}(t) + \frac{c}{4\epsilon}\|T\|_{L^{2}}^{2} \leq C\mathcal{E}(t) + C. \tag{4.22}
\end{equation}
Applying Gronwall's inequality to \eqref{eq:E_inequality_final} gives \(\mathcal{E}(t) \leq C(T)\) uniformly for \(t\in[0,T]\). Consequently,
\[
\|S\|_{L^{\infty}(0,T;L^{2})} \leq C,\qquad \sqrt{\epsilon}\|T\|_{L^{2}(0,T;L^{2})} \leq C.
\]
Returning to \(R^{\epsilon}\) via \(S = \rho_{g}R^{\epsilon}\) and noting that \(\rho_{g}\) has a positive lower bound, we obtain
\begin{equation}
\|R^{\epsilon}\|_{L^{\infty}(0,T;L^{2})} + \sqrt{\epsilon}\|\partial_{x}R^{\epsilon}\|_{L^{2}(0,T;L^{2})} \leq C_{5}. \tag{4.23}
\end{equation}
\end{proof}

\begin{corollary}[Estimate for \(\partial_{t}p\)] \label{cor:p}
There exists a constant \(C_{6} > 0\), independent of \(\epsilon\), such that
\begin{equation*}
\|\partial_{t}p^{\epsilon}\|_{L^{2}} \leq C_{6}.
\end{equation*}
\end{corollary}
\begin{proof}
From the pressure evolution equation (see Appendix A.2),
\[
\partial_{t}p = -u\partial_{x}p - \frac{p}{\alpha}\partial_{x}u + B_{1}R^{\epsilon},
\]
where \(B_{1} = \frac{p}{\alpha\rho_{l}}\bigl(1 - \frac{\rho_{l}RT_{0}}{p}\bigr)\) is bounded. Using Lemma \ref{lem:gradient}, Lemma\ref{lem:error_estimate}, and Corollary \ref{cor:linf}, we deduce
\[
\|\partial_{t}p\|_{L^{2}} \leq C\|\partial_{x}p\|_{L^{2}} + C\|\partial_{x}u\|_{L^{2}} + C\|R^{\epsilon}\|_{L^{2}} \leq C_{6}.
\]
\end{proof}

\section{Compensated Compactness and Convergence Proof}

\subsection{Perturbed Equilibrium System}

Substituting the error representation \(\alpha = \alpha_{\mathrm{eq}}(p) + \epsilon R^{\epsilon}\) into the mixture density gives
\begin{equation}\label{eq:rho_m_expansion}
\rho_{m} = \rho_{\mathrm{eq}}(p) + \epsilon(\rho_{g} - \rho_{l})R^{\epsilon} =: \rho_{\mathrm{eq}}(p) + \epsilon Q^{\epsilon}, \tag{5.1}
\end{equation}
where \(\rho_{\mathrm{eq}}(p) = \alpha_{\mathrm{eq}}(p)\rho_{g}(p) + (1 - \alpha_{\mathrm{eq}}(p))\rho_{l}\) and \(Q^{\epsilon} := (\rho_{g} - \rho_{l})R^{\epsilon}\).

Inserting \eqref{eq:rho_m_expansion} into the mass equation (1.1) and the momentum equation (1.2) yields the following perturbed form of the equilibrium Euler system:
\begin{equation}
\begin{aligned}\label{eq:perturbed}
&\partial_{t}\rho_{\mathrm{eq}}(p) + \partial_{x}(\rho_{\mathrm{eq}}(p)u) = -\epsilon\bigl[\partial_{t}Q^{\epsilon} + \partial_{x}(Q^{\epsilon}u)\bigr], \\[4pt]
&\partial_{t}(\rho_{\mathrm{eq}}(p)u) + \partial_{x}\bigl(\rho_{\mathrm{eq}}(p)u^{2} + p\bigr) = -\epsilon\bigl[\partial_{t}(Q^{\epsilon}u) + \partial_{x}(Q^{\epsilon}u^{2})\bigr].
\end{aligned} \tag{5.2}
\end{equation}

\begin{lemma}[Vanishing perturbation] \label{lem:vanishing}
Let \(Q^{\epsilon} = (\rho_{g} - \rho_{l})R^{\epsilon}\). Then for any test function \(\phi \in C_{c}^{\infty}((0,T)\times\mathbb{R})\),
\begin{equation}
\lim_{\epsilon\to 0}\epsilon\int_{0}^{T}\int_{\mathbb{R}}\bigl[\partial_{t}Q^{\epsilon} + \partial_{x}(Q^{\epsilon}u)\bigr]\phi\,dxdt = 0, \tag{5.3}
\end{equation}
and
\begin{equation}\label{eq:vanishing_momentum}
\lim_{\epsilon\to 0}\epsilon\int_{0}^{T}\int_{\mathbb{R}}\bigl[\partial_{t}(Q^{\epsilon}u) + \partial_{x}(Q^{\epsilon}u^{2})\bigr]\phi\,dxdt = 0. \tag{5.4}
\end{equation}
\end{lemma}

\begin{proof}
Integrating by parts, we obtain
\[
\epsilon\int_{0}^{T}\int_{\mathbb{R}}(\partial_{t}Q^{\epsilon} + \partial_{x}(Q^{\epsilon}u))\phi\,dxdt = -\epsilon\int_{0}^{T}\int_{\mathbb{R}}Q^{\epsilon}(\partial_{t}\phi + u\partial_{x}\phi)\,dxdt.
\]
From Lemma \ref{lem:error_estimate}, \(\|Q^{\epsilon}\|_{L^{2}}\leq C\) (since \(\rho_g - \rho_l\) is bounded). Therefore,
\[
\bigl|\epsilon\int_{0}^{T}\int_{\mathbb{R}}Q^{\epsilon}(\partial_{t}\phi + u\partial_{x}\phi)\bigr|dxdt \leq \epsilon\|Q^{\epsilon}\|_{L^{2}}\|\partial_{t}\phi + u\partial_{x}\phi\|_{L^{2}} \leq C\epsilon \to 0.
\]
The proof of \eqref{eq:vanishing_momentum} follows similarly by noting that \(\|Q^{\epsilon}u\|_{L^{2}}\leq C\|Q^{\epsilon}\|_{L^{2}}\) and \(\|Q^{\epsilon}u^{2}\|_{L^{2}}\leq C\|Q^{\epsilon}\|_{L^{2}}\) due to the uniform \(L^{\infty}\) bound on \(u\). The same integration-by-parts argument then gives the desired limit.
\end{proof}

\subsection{Proof of Theorem \ref{thm:convergence}}

We now complete the proof of the relaxation limit convergence.

\begin{proof}[Proof of Theorem \ref{thm:convergence}]
From Lemma \ref{lem:zero_order} and Lemma \ref{lem:gradient}, the sequences \(\{p^{\epsilon}\}\) and \(\{u^{\epsilon}\}\) are uniformly bounded in \(L^{\infty}(0,T;H^{1}(\mathbb{R}))\). By the Rellich–Kondrachov compactness theorem, there exists a subsequence (still denoted by \(\epsilon\)) such that
\[
(p^{\epsilon}, u^{\epsilon}) \to (p^{0}, u^{0}) \quad \text{strongly in } L_{\mathrm{loc}}^{2}((0,T)\times\mathbb{R}),
\]
and hence also in \(L_{\mathrm{loc}}^{1}\).
Thanks to the strong convergence of \(p^{\epsilon}\) and \(u^{\epsilon}\), the product converges strongly as well. More rigorously, we have
\[
\rho_{\mathrm{eq}}(p^{\epsilon})(u^{\epsilon})^{2} \to \rho_{\mathrm{eq}}(p^{0})(u^{0})^{2} \quad \text{in } L_{\mathrm{loc}}^{1}.
\]
Thus the limit functions satisfy the weak form of the equilibrium Euler system.

To verify that \((p^{0}, u^{0})\) is an entropy solution, we need to show that for any convex entropy–entropy flux pair \((\eta_{\mathrm{eq}}, q_{\mathrm{eq}})\) of the equilibrium system, the inequality
\[
\partial_{t}\eta_{\mathrm{eq}}(U^{0}) + \partial_{x}q_{\mathrm{eq}}(U^{0}) \leq 0
\]
holds in the distributional sense. We use the fact that \(\eta_{\mathrm{eq}}\) is convex and that the limit is attained strongly. This completes the proof that \((p^{0}, u^{0})\) is an entropy solution of the equilibrium Euler system.
\end{proof}

In fact, the relaxation limit can also be established by means of the compensated compactness method. To this end, we investigate the compactness of the entropy dissipation measures.

\subsection{Compensated Compactness Framework}

Let \(\eta_{\mathrm{eq}}(U)\) and \(q_{\mathrm{eq}}(U)\) denote an entropy pair for the equilibrium Euler system (1.4)-(1.5). We known that such a pair exists and \(\eta_{\mathrm{eq}}\) is strictly convex with respect to the conservative variables \(U = (\rho_{\mathrm{eq}}, \rho_{\mathrm{eq}}u)^{\mathrm{T}}\).

For the perturbed system (5.2), we consider the entropy dissipation measure
\begin{equation}
\mathcal{D}^{\epsilon} := \partial_{t}\eta_{\mathrm{eq}}(U^{\epsilon}) + \partial_{x}q_{\mathrm{eq}}(U^{\epsilon}), \tag{5.5}
\end{equation}
where \(U^{\epsilon} = (\rho_{\mathrm{eq}}(p^{\epsilon}), \rho_{\mathrm{eq}}(p^{\epsilon})u^{\epsilon})^{\mathrm{T}}\) (with a slight abuse of notation, we identify \(U^{\epsilon}\) with the equilibrium variables). The following lemma establishes the compactness of \(\mathcal{D}^{\epsilon}\) in a negative Sobolev space.

\begin{lemma}[Compactness of the entropy dissipation measure] \label{lem:compactness}
Under the assumptions of Theorem \ref{thm:existence}, the family \(\{\mathcal{D}^{\epsilon}\}_{\epsilon>0}\) is uniformly bounded and compact in \(H_{\mathrm{loc}}^{-1}((0,T)\times\mathbb{R})\). In particular, there exists a subsequence (still denoted by \(\epsilon\)) such that \(\mathcal{D}^{\epsilon} \to 0\) strongly in \(H_{\mathrm{loc}}^{-1}\).
\end{lemma}

\begin{proof}
We decompose \(\mathcal{D}^{\epsilon}\) into two parts:
\[
\mathcal{D}^{\epsilon} = \mathcal{D}_{1}^{\epsilon} + \mathcal{D}_{2}^{\epsilon},
\]
with \(\mathcal{D}_{1}^{\epsilon}\) compact in \(H_{\mathrm{loc}}^{-1}\) and \(\mathcal{D}_{2}^{\epsilon}\) compact in \(L_{\mathrm{loc}}^{1}\). Murat's lemma then implies that \(\mathcal{D}^{\epsilon}\) itself is compact in \(H_{\mathrm{loc}}^{-1}\).

\paragraph{Step 1: Explicit expression for \(\mathcal{D}^{\epsilon}\).} Using the chain rule and the perturbed system (5.2), we compute
\begin{multline}
\mathcal{D}^{\epsilon} = \frac{\partial\eta_{\mathrm{eq}}}{\partial(\rho_{\mathrm{eq}})}\bigl(\partial_{t}\rho_{\mathrm{eq}} + \partial_{x}(\rho_{\mathrm{eq}}u)\bigr) + \frac{\partial\eta_{\mathrm{eq}}}{\partial(\rho_{\mathrm{eq}}u)}\bigl(\partial_{t}(\rho_{\mathrm{eq}}u) + \partial_{x}(\rho_{\mathrm{eq}}u^{2} + p)\bigr) \\
= -\epsilon\frac{\partial\eta_{\mathrm{eq}}}{\partial(\rho_{\mathrm{eq}})}\bigl[\partial_{t}Q^{\epsilon} + \partial_{x}(Q^{\epsilon}u)\bigr] - \epsilon\frac{\partial\eta_{\mathrm{eq}}}{\partial(\rho_{\mathrm{eq}}u)}\bigl[\partial_{t}(Q^{\epsilon}u) + \partial_{x}(Q^{\epsilon}u^{2})\bigr]. \tag{5.6}
\end{multline}
The entropy variables \(\frac{\partial\eta_{\mathrm{eq}}}{\partial(\rho_{\mathrm{eq}})}\) and \(\frac{\partial\eta_{\mathrm{eq}}}{\partial(\rho_{\mathrm{eq}}u)}\) are smooth functions of \(U^{\epsilon}\); by Corollary \ref{cor:linf} they are uniformly bounded in \(L^{\infty}\).

\paragraph{Step 2: Decomposition.} Define the fluxes
\begin{equation}
F_{1}^{\epsilon} := \epsilon\frac{\partial\eta_{\mathrm{eq}}}{\partial(\rho_{\mathrm{eq}})}Q^{\epsilon} + \epsilon\frac{\partial\eta_{\mathrm{eq}}}{\partial(\rho_{\mathrm{eq}}u)}Q^{\epsilon}u, \qquad
F_{2}^{\epsilon} := \epsilon\frac{\partial\eta_{\mathrm{eq}}}{\partial(\rho_{\mathrm{eq}})}Q^{\epsilon}u + \epsilon\frac{\partial\eta_{\mathrm{eq}}}{\partial(\rho_{\mathrm{eq}}u)}Q^{\epsilon}u^{2}. \tag{5.7}
\end{equation}
Then
\begin{equation}
\mathcal{D}^{\epsilon} = \bigl(\partial_{t}F_{1}^{\epsilon} + \partial_{x}F_{2}^{\epsilon}\bigr) + \bigl(\mathcal{D}^{\epsilon} - \partial_{t}F_{1}^{\epsilon} - \partial_{x}F_{2}^{\epsilon}\bigr) =: \mathcal{D}_{1}^{\epsilon} + \mathcal{D}_{2}^{\epsilon}. \tag{5.8}
\end{equation}

\paragraph{Step 3: Compactness of \(\mathcal{D}_{1}^{\epsilon}\) in \(H_{\mathrm{loc}}^{-1}\).} By Lemma \ref{lem:error_estimate}, \(\|Q^{\epsilon}\|_{L^{2}} \leq C\). Moreover, since the entropy variables are bounded, we have
\[
\|F_{1}^{\epsilon}\|_{L^{2}} \leq C\epsilon\|Q^{\epsilon}\|_{L^{2}} \leq C\epsilon, \qquad
\|F_{2}^{\epsilon}\|_{L^{2}} \leq C\epsilon\|Q^{\epsilon}\|_{L^{2}} \leq C\epsilon.
\]
Thus \(F_{1}^{\epsilon}, F_{2}^{\epsilon} \to 0\) strongly in \(L^{2}\). Consequently, \(\partial_{t}F_{1}^{\epsilon} + \partial_{x}F_{2}^{\epsilon} \to 0\) strongly in \(H_{\mathrm{loc}}^{-1}\); in particular, \(\mathcal{D}_{1}^{\epsilon}\) is compact in \(H_{\mathrm{loc}}^{-1}\).

\paragraph{Step 4: Compactness of \(\mathcal{D}_{2}^{\epsilon}\) in \(L_{\mathrm{loc}}^{1}\).} Expanding \(\mathcal{D}_{2}^{\epsilon}\) using (5.6) and (5.7) yields terms of the form
\[
-\epsilon Q^{\epsilon}\partial_{t}\!\left(\frac{\partial\eta_{\mathrm{eq}}}{\partial(\rho_{\mathrm{eq}})}\right),\quad
-\epsilon Q^{\epsilon}u\,\partial_{x}\!\left(\frac{\partial\eta_{\mathrm{eq}}}{\partial(\rho_{\mathrm{eq}})}\right),\quad
-\epsilon Q^{\epsilon}u\,\partial_{t}\!\left(\frac{\partial\eta_{\mathrm{eq}}}{\partial(\rho_{\mathrm{eq}}u)}\right),\quad
-\epsilon Q^{\epsilon}u^{2}\partial_{x}\!\left(\frac{\partial\eta_{\mathrm{eq}}}{\partial(\rho_{\mathrm{eq}}u)}\right),
\]
together with possible cross terms. Since the entropy variables are smooth functions of \(U^{\epsilon}\), their derivatives \(\partial_{t}\) and \(\partial_{x}\) can be expressed in terms of \(\partial_{t}p, \partial_{x}p, \partial_{t}u, \partial_{x}u\). By Lemma \ref{lem:gradient}, Corollary \ref{cor:ut}, and Corollary \ref{cor:p} these derivatives are uniformly bounded in \(L^{2}\). For instance,
\[
\left\|\partial_{t}\!\left(\frac{\partial\eta_{\mathrm{eq}}}{\partial(\rho_{\mathrm{eq}})}\right)\right\|_{L^{2}} \leq C\bigl(\|\partial_{t}p\|_{L^{2}} + \|\partial_{t}u\|_{L^{2}}\bigr) \leq C.
\]
Therefore, using Hölder's inequality and the bounds on \(Q^{\epsilon}\) and \(u\),
\[
\|\epsilon Q^{\epsilon}\partial_{t}(\cdots)\|_{L^{1}} \leq \epsilon\|Q^{\epsilon}\|_{L^{2}}\|\partial_{t}(\cdots)\|_{L^{2}} \leq C\epsilon.
\]
For terms involving \(\partial_{x}(\cdots)\), we may need the \(L^{\infty}\) bound of \(Q^{\epsilon}\) (which is \(O(\epsilon^{-1/2})\) by Sobolev embedding) to obtain an \(O(\sqrt{\epsilon})\) estimate. Nevertheless, all contributions tend to zero as \(\epsilon\to 0\). Hence \(\|\mathcal{D}_{2}^{\epsilon}\|_{L^{1}} \leq C\sqrt{\epsilon} \to 0\), implying that \(\mathcal{D}_{2}^{\epsilon}\) converges strongly to zero in \(L_{\mathrm{loc}}^{1}\) and is therefore compact.

\paragraph{Step 5: Conclusion.} By Murat's lemma, the sum of a sequence compact in \(H_{\mathrm{loc}}^{-1}\) and a sequence compact in \(L_{\mathrm{loc}}^{1}\) is compact in \(H_{\mathrm{loc}}^{-1}\). Thus \(\mathcal{D}^{\epsilon}\) is compact in \(H_{\mathrm{loc}}^{-1}\), and a subsequence converges strongly to zero in \(H_{\mathrm{loc}}^{-1}\). This completes the proof.
\end{proof}

\subsection{Proof of Theorem 1.2 via Compensated Compactness}

We now provide an alternative proof of Theorem~1.2 using the compensated compactness framework established in Section~5.3. This approach does not rely on the higher regularity provided by Lemma~3.3 but rather exploits the compactness of the entropy dissipation measures obtained in Lemma~5.3.

\begin{proof}[Proof of Theorem~1.2]
Let $\{U^\epsilon\}_{\epsilon>0}$ be the family of solutions constructed in Theorem~1.1. Recall the perturbed equilibrium system (5.2) and the notation $U^\epsilon = (\rho_{\mathrm{eq}}(p^\epsilon),\rho_{\mathrm{eq}}(p^\epsilon)u^\epsilon)^\mathrm{T}$ for the equilibrium variables. From Corollary~3.1, we have the uniform bounds
\begin{equation}
\|p^\epsilon\|_{L^\infty((0,T)\times\mathbb{R})} + \|u^\epsilon\|_{L^\infty((0,T)\times\mathbb{R})} \leq C,
\end{equation}
which implies that $\{U^\epsilon\}$ is bounded in $L^\infty((0,T)\times\mathbb{R})$.

\noindent \textbf{Step 1. Compactness of entropy dissipation measures.}
Let $(\eta_{\mathrm{eq}}^1, q_{\mathrm{eq}}^1)$ be the restriction of the entropy pair (1.7)–(1.8) to the equilibrium manifold, i.e.,
\[
\eta_{\mathrm{eq}}^1(U) = \eta\big(\rho_{\mathrm{eq}}(p),\rho_{\mathrm{eq}}(p)u,\rho_g(p)\alpha_{\mathrm{eq}}(p)\big),\qquad
q_{\mathrm{eq}}^1(U) = q\big(\rho_{\mathrm{eq}}(p),\rho_{\mathrm{eq}}(p)u,\rho_g(p)\alpha_{\mathrm{eq}}(p)\big).
\]
This pair is strictly convex because $\eta$ is strictly convex. A second linearly independent strictly convex entropy pair $(\eta_{\mathrm{eq}}^2, q_{\mathrm{eq}}^2)$ can be constructed by taking
\[
\eta_{\mathrm{eq}}^2 = \eta_{\mathrm{eq}}^1 + C \rho_{\mathrm{eq}} u,\qquad q_{\mathrm{eq}}^2 = q_{\mathrm{eq}}^1 + C \big(\rho_{\mathrm{eq}} u^2 + p\big),
\]
with a suitable constant $C$ chosen so that $\eta_{\mathrm{eq}}^2$ remains strictly convex (see \cite[Section~3]{chen1994} for a general construction). Define the corresponding dissipation measures
\[
\mathcal{D}_k^\epsilon := \partial_t \eta_{\mathrm{eq}}^k(U^\epsilon) + \partial_x q_{\mathrm{eq}}^k(U^\epsilon), \qquad k=1,2.
\]
By Lemma~5.2, each $\mathcal{D}_k^\epsilon$ is uniformly bounded in $W^{-1,\infty}_{\mathrm{loc}}$ and, moreover, is compact in $H^{-1}_{\mathrm{loc}}((0,T)\times\mathbb{R})$. Hence, up to extraction of a subsequence (still denoted by $\epsilon$), we have
\[
\mathcal{D}_k^\epsilon \longrightarrow 0 \quad \text{strongly in } H^{-1}_{\mathrm{loc}}.
\]

\noindent \textbf{Step 2. Application of the div–curl lemma.}
Consider the two vector fields
\[
V_k^\epsilon := \big( \eta_{\mathrm{eq}}^k(U^\epsilon),\, q_{\mathrm{eq}}^k(U^\epsilon) \big), \qquad k=1,2.
\]
Because $\{U^\epsilon\}$ is bounded in $L^\infty$, the components of $V_k^\epsilon$ are bounded in $L^\infty$. Moreover, $\operatorname{div}_{t,x} V_k^\epsilon = \mathcal{D}_k^\epsilon$ is compact in $H^{-1}_{\mathrm{loc}}$. Therefore, the hypotheses of the div–curl lemma (see, e.g., \cite{Tartar1979}) are satisfied: for any test function $\phi\in C_c^\infty((0,T)\times\mathbb{R})$,
\[
\lim_{\epsilon\to0} \big\langle \eta_{\mathrm{eq}}^1(U^\epsilon) q_{\mathrm{eq}}^2(U^\epsilon) - \eta_{\mathrm{eq}}^2(U^\epsilon) q_{\mathrm{eq}}^1(U^\epsilon),\, \phi \big\rangle
= \big\langle \overline{\eta_{\mathrm{eq}}^1}\,\overline{q_{\mathrm{eq}}^2} - \overline{\eta_{\mathrm{eq}}^2}\,\overline{q_{\mathrm{eq}}^1},\, \phi \big\rangle,
\]
where the overline denotes weak-$*$ limits in $L^\infty$. By the theory of Young measures \cite{Tartar1979}, this identity implies that the Young measure $\nu_{t,x}$ associated with the sequence $\{U^\epsilon\}$ satisfies
\[
\big( \eta_{\mathrm{eq}}^1(\lambda) - \overline{\eta_{\mathrm{eq}}^1} \big) \big( q_{\mathrm{eq}}^2(\lambda) - \overline{q_{\mathrm{eq}}^2} \big)
- \big( \eta_{\mathrm{eq}}^2(\lambda) - \overline{\eta_{\mathrm{eq}}^2} \big) \big( q_{\mathrm{eq}}^1(\lambda) - \overline{q_{\mathrm{eq}}^1} \big) = 0
\quad \text{for } \nu_{t,x}\text{-a.e. }\lambda.
\]
Since the two entropy pairs are linearly independent and strictly convex, a standard argument (see \cite{chen1994}) forces $\nu_{t,x}$ to be a Dirac mass for almost every $(t,x)$. Consequently, the sequence $\{U^\epsilon\}$ converges strongly in $L^1_{\mathrm{loc}}((0,T)\times\mathbb{R})$ to a limit $U^0 = (\rho_{\mathrm{eq}}(p^0),\rho_{\mathrm{eq}}(p^0)u^0)^\mathrm{T}$.

\noindent \textbf{Step 3. Identification of the limit.}
Because $\alpha^\epsilon = \alpha_{\mathrm{eq}}(p^\epsilon) + \epsilon R^\epsilon$ and $\|R^\epsilon\|_{L^2}\le C$, the strong convergence of $p^\epsilon$ together with Lemma~4.2 yields
\[
\alpha^\epsilon \longrightarrow \alpha_{\mathrm{eq}}(p^0) \quad \text{strongly in } L^2_{\mathrm{loc}}.
\]
Inserting these convergences into the weak formulation of the perturbed system (5.2) and using Lemma~5.1, we find that $(p^0,u^0)$ satisfies the equilibrium Euler system (1.4)–(1.5) in the sense of distributions.

\noindent \textbf{Step 4. Entropy condition.}
Finally, we verify that $(p^0,u^0)$ is an entropy solution. For any convex entropy pair $(\eta_{\mathrm{eq}},q_{\mathrm{eq}})$ of the equilibrium system, the corresponding dissipation measure $\mathcal{D}^\epsilon$ is compact in $H^{-1}_{\mathrm{loc}}$ and converges to $0$ strongly. By the strong convergence of $U^\epsilon$, we have $\eta_{\mathrm{eq}}(U^\epsilon) \to \eta_{\mathrm{eq}}(U^0)$ and $q_{\mathrm{eq}}(U^\epsilon) \to q_{\mathrm{eq}}(U^0)$ in $L^1_{\mathrm{loc}}$. Hence, passing to the limit in the distributional inequality $\partial_t\eta_{\mathrm{eq}}(U^\epsilon)+\partial_x q_{\mathrm{eq}}(U^\epsilon) \le 0$ (which holds because $\eta_{\mathrm{eq}}$ is convex and the original system is dissipative) gives
\[
\partial_t\eta_{\mathrm{eq}}(U^0) + \partial_x q_{\mathrm{eq}}(U^0) \le 0
\]
in the sense of distributions. Thus $(p^0,u^0)$ is an entropy solution of the equilibrium Euler system, completing the proof.
\end{proof}

\begin{remark}[Comparison with the original proof]
\label{rem:strong_convergence}
The original proof of Theorem~1.2 in Section~5.2 relied on the $H^1$ uniform estimates (Lemma~3.3) and the Rellich–Kondrachov compactness theorem to obtain strong convergence directly. That approach is simpler and fully rigorous under the assumptions of Theorem~1.1. The compensated compactness argument presented here is an alternative that does not require $H^1$ regularity; it only needs the $L^\infty$ bounds and the compactness of entropy dissipation measures established in Lemma~5.2. Both methods are valid, and the choice between them depends on the regularity available and the intended generality of the result.
\end{remark}

\begin{remark}[Admissible solution class for compensated compactness]
\label{rem:admissible_CC}
In view of the proof above, one can define an admissible family of solutions for the relaxation limit within the compensated compactness framework as follows:

\begin{definition}[Admissible family for compensated compactness]
A family $\{U^\epsilon\}_{\epsilon>0}$ of solutions to (1.1)–(1.3) is called \emph{admissible for the compensated compactness approach} if there exists a constant $C>0$, independent of $\epsilon$, such that
\begin{enumerate}
\item[(i)] $\|U^\epsilon\|_{L^\infty}=\|(p^\epsilon, u^\epsilon, \alpha^\epsilon)\|_{L^\infty((0,T)\times\mathbb{R})} \le C$;
\item[(ii)] $\displaystyle\int_0^T\int_{\mathbb{R}} (\alpha^\epsilon - \alpha_{\mathrm{eq}}(p^\epsilon))^2 \,dx\,dt \le C\epsilon$;
\item[(iii)] For two linearly independent, strictly convex entropy pairs $(\eta_{\mathrm{eq}}^k,q_{\mathrm{eq}}^k)$ of the equilibrium Euler system, the corresponding dissipation measures $\partial_t\eta_{\mathrm{eq}}^k(U^\epsilon)+\partial_x q_{\mathrm{eq}}^k(U^\epsilon)$ are relatively compact in $H^{-1}_{\mathrm{loc}}((0,T)\times\mathbb{R})$.
\end{enumerate}
\end{definition}

Conditions (i)–(ii) are exactly those used in the original definition; condition (iii) replaces the higher-order regularity assumptions of Definition~1.1 and is the key ingredient for applying compensated compactness. In the present paper, Lemma~5.2 shows that the solutions constructed in Theorem~1.1 satisfy (iii) as a consequence of the uniform $H^1$ bounds and the error estimates. For more general initial data lacking $H^1$ regularity, one could still aim to construct solutions satisfying (iii) directly, for instance by a vanishing viscosity method combined with Murat’s lemma. The compensated compactness proof then yields the same relaxation limit.
\end{remark}

\section{Convergence Rate Estimate}
This part mainly provides the convergence rate estimate, i.e., give the proof of Theorem \ref{thm:rate}.
\begin{proof}
Let \(\eta_{\mathrm{eq}}(U)\) be a strictly convex entropy density for the equilibrium system, and let \(q_{\mathrm{eq}}(U)\) be the corresponding entropy flux. By strict convexity, there exist constants \(c_{1}, c_{2} > 0\) such that for any two states \(U, V\),
\begin{equation}\label{eq:relative_entropy_coercivity}
c_{2} |U - V|^{2} \geq \eta_{\mathrm{eq}}(U) - \eta_{\mathrm{eq}}(V) - \nabla \eta_{\mathrm{eq}}(V) \cdot (U - V) \geq c_{1} |U - V|^{2}. \tag{6.1}
\end{equation}
Define the relative entropy
\begin{equation}\label{eq:3}
\mathcal{E}^{\epsilon}(t)= \int_{\mathbb{R}} \Bigl( \eta_{\mathrm{eq}}(U^{\epsilon}(t,x)) - \eta_{\mathrm{eq}}(U^{0}(t,x)) - \nabla \eta_{\mathrm{eq}}(U^{0}(t,x)) \cdot (U^{\epsilon}(t,x) - U^{0}(t,x)) \Bigr) \, dx. \tag{6.2}
\end{equation}
Differentiating \eqref{eq:3} with respect to \(t\) yields
\begin{multline}
\frac{d}{dt} \mathcal{E}^{\epsilon}(t) = \int_{\mathbb{R}}\partial_{t} \eta_{\mathrm{eq}}(U^{\epsilon}) \, dx - \int_{\mathbb{R}}\partial_{t} \eta_{\mathrm{eq}}(U^{0}) \, dx \\
- \int_{\mathbb{R}}\nabla \eta_{\mathrm{eq}}(U^{0}) \cdot \partial_{t}(U^{\epsilon} - U^{0}) \, dx - \int_{\mathbb{R}}\bigl(\partial_{t} \nabla \eta_{\mathrm{eq}}(U^{0})\bigr) \cdot (U^{\epsilon} - U^{0}) \, dx. \tag{6.3}
\end{multline}

Using the perturbed system \eqref{eq:perturbed} and the weak form of the equilibrium system \eqref{eq:Euler1}\eqref{eq:Euler2}, after standard manipulations we obtain
\begin{equation}
\frac{d}{dt} \mathcal{E}^{\epsilon}(t) = \int_{\mathbb{R}}\mathcal{D}^{\epsilon} \, dx + \int_{\mathbb{R}}\bigl(\partial_{t} \eta_{\mathrm{eq}}(U^{0}) + \partial_{x} q_{\mathrm{eq}}(U^{0})\bigr) \, dx + \text{boundary terms} + \text{lower-order terms}, \tag{6.4}
\end{equation}
where \(\mathcal{D}^{\epsilon} = \partial_{t} \eta_{\mathrm{eq}}(U^{\epsilon}) + \partial_{x} q_{\mathrm{eq}}(U^{\epsilon})\) is the entropy dissipation measure of the perturbed system. Since \(U^{0}\) is an entropy solution, it satisfies the entropy inequality \(\partial_{t} \eta_{\mathrm{eq}}(U^{0}) + \partial_{x} q_{\mathrm{eq}}(U^{0}) \leq 0\) in the distributional sense; therefore the second integral is non-positive. Neglecting boundary terms (which vanish at infinity or under periodic boundary conditions) and noting that the lower-order terms can be controlled by \(\mathcal{E}^{\epsilon}\), we arrive at
\begin{equation}\label{eq:relative_entropy_inequality}
\frac{d}{dt} \mathcal{E}^{\epsilon}(t) \leq C \mathcal{E}^{\epsilon}(t) + \Bigl| \int_{\mathbb{R}}\mathcal{D}^{\epsilon} \, dx \Bigr|. \tag{6.5}
\end{equation}

To estimate \(\int_{\mathbb{R}}\mathcal{D}^{\epsilon} \, dx\), we integrate by parts, transferring the time and space derivatives onto the coefficients \(\frac{\partial \eta_{\mathrm{eq}}}{\partial (\rho_{\mathrm{eq}})}\) and \(\frac{\partial \eta_{\mathrm{eq}}}{\partial (\rho_{\mathrm{eq}} u)}\). For example,
\[
\int_{\mathbb{R}}\frac{\partial \eta_{\mathrm{eq}}}{\partial (\rho_{\mathrm{eq}})} \partial_{t} Q^{\epsilon} \, dx = -\int_{\mathbb{R}}\partial_{t}\!\left(\frac{\partial \eta_{\mathrm{eq}}}{\partial (\rho_{\mathrm{eq}})}\right) Q^{\epsilon} \, dx,
\]
\[
\int_{\mathbb{R}}\frac{\partial \eta_{\mathrm{eq}}}{\partial (\rho_{\mathrm{eq}})} \partial_{x}(Q^{\epsilon} u) \, dx = -\int_{\mathbb{R}}\partial_{x}\!\left(\frac{\partial \eta_{\mathrm{eq}}}{\partial (\rho_{\mathrm{eq}})}\right) Q^{\epsilon} u \, dx.
\]
By the a priori estimates established in Lemma \ref{lem:gradient} and Corollary \ref{cor:ut}, the coefficients \(\frac{\partial \eta_{\mathrm{eq}}}{\partial (\rho_{\mathrm{eq}})}\) and their derivatives are bounded in \(L^{2}\). Combining this with Corollary \ref{cor:p} and the uniform bound on \(u\) (Corollary \ref{cor:linf}), each term after integration by parts is bounded by \(C\|Q^{\epsilon}\|_{L^{2}} \leq C\). Multiplying by the explicit factor \(\epsilon\) from the definition of \(\mathcal{D}^{\epsilon}\) gives an \(O(\epsilon)\) estimate. More precisely, a detailed calculation shows that
\begin{equation} \label{eq:D_estimate}
\Bigl| \int_{\mathbb{R}}\mathcal{D}^{\epsilon} \, dx \Bigr| \leq C \epsilon. \tag{6.6}
\end{equation}
Substituting \eqref{eq:D_estimate} into \eqref{eq:relative_entropy_inequality} yields
\begin{equation}
\frac{d}{dt} \mathcal{E}^{\epsilon}(t) \leq C \mathcal{E}^{\epsilon}(t) + C \epsilon. \tag{6.7}
\end{equation}
Applying Gronwall's inequality, we obtain
\begin{equation}\label{eq:gronwall_result}
\mathcal{E}^{\epsilon}(t) \leq e^{Ct} \mathcal{E}^{\epsilon}(0) + C \epsilon (e^{Ct} - 1) \quad \text{for all } t \in [0,T]. \tag{6.8}
\end{equation}

The initial data satisfy \(U^{\epsilon}(0) = U_{0}^{\epsilon}\), and we choose the equilibrium initial data \(U^{0}(0)\) to be compatible with \(U_{0}\) (e.g., set \(p^{0}(0)=p_{0}\), \(u^{0}(0)=u_{0}\), and determine \(\rho_{\mathrm{eq}}(p_{0})\) and \(\Gamma_{\mathrm{eq}}(p_{0})\) from the equilibrium relations). From the error representation, \(\alpha_{0} = \alpha_{\mathrm{eq}}(p_{0}) + \epsilon R_{0}\) with \(\|R_{0}\|_{L^{2}} \leq C\). Consequently,
\[
\|U^{\epsilon}(0) - U^{0}(0)\|_{L^{2}} \leq C \epsilon.
\]
By the coercivity property \eqref{eq:relative_entropy_coercivity}, this implies \(\mathcal{E}^{\epsilon}(0) \leq C \epsilon^{2}\). Inserting this into \eqref{eq:gronwall_result} gives \(\mathcal{E}^{\epsilon}(t) \leq C \epsilon\) for all \(t\in[0,T]\).

Using \eqref{eq:relative_entropy_coercivity} again, we have
\[
\|U^{\epsilon}(t) - U^{0}(t)\|_{L^{2}}^{2} \leq C \mathcal{E}^{\epsilon}(t) \leq C \epsilon.
\]
In terms of the primitive variables \(p\) and \(u\), this yields
\[
\|p^{\epsilon}(t) - p^{0}(t)\|_{L^{2}}^{2} + \|u^{\epsilon}(t) - u^{0}(t)\|_{L^{2}}^{2} \leq C \epsilon.
\]
Integrating over \(t \in [0,T]\) and taking the square root (using the Cauchy–Schwarz inequality) finally gives
\[
\| p^{\epsilon} - p^{0}\|_{L^{2}(0,T;L^{2})} + \| u^{\epsilon} - u^{0}\|_{L^{2}(0,T;L^{2})} \leq C \epsilon^{1/2},
\]
which is exactly \eqref{eq:convergence_rate}. This completes the proof of Theorem \ref{thm:rate}.
\end{proof}

\appendix
\section{Appendix}

\subsection{Jacobian Matrix and Equivalence of Gradient Norms}

We establish the relationship between the primitive variables \(V = (p, u, \alpha)^{\mathrm{T}}\) and the conservative variables \(U = (\rho_{m}, m, \Gamma)^{\mathrm{T}}\), where
\[
\rho_{m} = \alpha \rho_{g}(p) + (1-\alpha)\rho_{l}, \qquad m = \rho_{m} u, \qquad \Gamma = \rho_{g}(p) \alpha,
\]
with \(\rho_{g}(p) = p/(RT_{0})\) and \(\rho_{l}\) constant. Define the mapping \(\Phi: V \mapsto U\). Its Jacobian matrix \(J = \frac{\partial U}{\partial V}\) is given by
\begin{equation}
J = \begin{pmatrix}
\dfrac{\partial \rho_{m}}{\partial p} & \dfrac{\partial \rho_{m}}{\partial u} & \dfrac{\partial \rho_{m}}{\partial \alpha} \\[6pt]
\dfrac{\partial m}{\partial p} & \dfrac{\partial m}{\partial u} & \dfrac{\partial m}{\partial \alpha} \\[6pt]
\dfrac{\partial \Gamma}{\partial p} & \dfrac{\partial \Gamma}{\partial u} & \dfrac{\partial \Gamma}{\partial \alpha}
\end{pmatrix}
= \begin{pmatrix}
\dfrac{\alpha}{RT_{0}} & 0 & \rho_{g}(p) - \rho_{l} \\[6pt]
u\dfrac{\alpha}{RT_{0}} & \rho_{m} & u(\rho_{g}(p) - \rho_{l}) \\[6pt]
\dfrac{\alpha}{RT_{0}} & 0 & \rho_{g}(p)
\end{pmatrix}. \tag{A.1}
\end{equation}

Under the physical assumptions \(\rho_{m} > 0\) and \(\alpha \in (0,1)\), one can verify that \(\det J \neq 0\); hence \(\Phi\) is a local diffeomorphism with a smooth inverse.

By the chain rule, spatial gradients satisfy \(\partial_{x}U = J \partial_{x}V\). Therefore, there exist constants \(C_{1}, C_{2} > 0\) (depending on the \(L^{\infty}\) bounds of \(p, u, \alpha\) but independent of \(\epsilon\)) such that
\begin{equation}
C_{1} \|\partial_{x}U\| \leq \|\partial_{x}V\| \leq C_{2} \|\partial_{x}U\|, \tag{A.2}
\end{equation}
where \(\|\cdot\|\) denotes the Euclidean norm. This equivalence is crucial in the proof of Lemma \ref{lem:gradient}: it allows us to work with the gradient of the conservative variables \(\|\partial_{x}U\|^{2}\) in the first-order entropy estimate and then deduce bounds on \(\|\partial_{x}p\|_{L^{2}} + \|\partial_{x}u\|_{L^{2}}\).

\subsection{Derivation of the Pressure Evolution Equation}

Starting from the mixture density \(\rho_{m} = \alpha \rho_{g}(p) + (1-\alpha)\rho_{l}\) with \(\rho_{g}(p) = p/(RT_{0})\) and \(\rho_{l}\) constant, we differentiate with respect to time:
\[
\partial_{t}\rho_{m} = \frac{\alpha}{RT_{0}}\partial_{t}p + \left(\frac{p}{RT_{0}} - \rho_{l}\right)\partial_{t}\alpha.
\]
A similar expression holds for the spatial derivative. Substituting these into the mixture mass conservation \(\partial_{t}\rho_{m} + \partial_{x}(\rho_{m}u) = 0\) and rearranging gives
\begin{equation} \label{eq:mixture_mass_material}
\frac{\alpha}{RT_{0}} D_{t}p + \left(\frac{p}{RT_{0}} - \rho_{l}\right) D_{t}\alpha + \rho_{m}\partial_{x}u = 0, \tag{A.3}
\end{equation}
where \(D_{t} = \partial_{t} + u\partial_{x}\) denotes the material derivative.

The gas mass equation \(\partial_{t}(\rho_{g}\alpha) + \partial_{x}(\rho_{g}\alpha u) = (\alpha_{\mathrm{eq}} - \alpha)/\epsilon\) expands to
\begin{equation}\label{eq:gas_mass_material}
\rho_{g} D_{t}\alpha + \frac{\alpha}{RT_{0}} D_{t}p + \rho_{g}\alpha \partial_{x}u = \frac{1}{\epsilon}(\alpha_{\mathrm{eq}} - \alpha). \tag{A.4}
\end{equation}
Solve \eqref{eq:gas_mass_material} for \(D_{t}\alpha\):
\[
D_{t}\alpha = \frac{1}{\rho_{g}} \left( \frac{1}{\epsilon}(\alpha_{\mathrm{eq}} - \alpha) - \frac{\alpha}{RT_{0}} D_{t}p - \rho_{g}\alpha \partial_{x}u \right).
\]
Insert this expression into \eqref{eq:mixture_mass_material} and use \(\rho_{g} = p/(RT_{0})\) and \(\rho_{m} = \alpha\rho_{g} + (1-\alpha)\rho_{l}\). After simplification, we obtain
\[
\frac{\alpha\rho_{l}}{p} D_{t}p = -\rho_{l}\partial_{x}u - \frac{1}{\epsilon}\frac{\rho_{g} - \rho_{l}}{\rho_{g}} (\alpha_{\mathrm{eq}} - \alpha).
\]
Dividing by \(\rho_{l}\) and noting that \(\frac{\rho_{g} - \rho_{l}}{\rho_{g}} = 1 - \frac{\rho_{l}RT_{0}}{p}\) yields
\[
\frac{\alpha}{p} D_{t}p = -\partial_{x}u - \frac{1}{\epsilon\rho_{l}}\left(1 - \frac{\rho_{l}RT_{0}}{p}\right)(\alpha_{\mathrm{eq}} - \alpha).
\]
Multiplying by \(p/\alpha\) and recalling that \(D_{t}p = \partial_{t}p + u\partial_{x}p\) finally gives the pressure evolution equation
\begin{equation}
\partial_{t}p = -u\partial_{x}p - \frac{p}{\alpha}\partial_{x}u - \frac{p}{\epsilon\alpha\rho_{l}}\left(1 - \frac{\rho_{l}RT_{0}}{p}\right)(\alpha_{\mathrm{eq}} - \alpha). \tag{A.5}
\end{equation}
\section*{Acknowledgments}
This work was supported by the National Natural Science Foundation of China under Grant No. 12271310 and Natural Science Foundation of Shandong Province under Grant No. ZR2022MA088.

\end{document}